\documentclass[11pt,a4paper]{article}

\usepackage{epsf,epsfig,amsfonts,amsgen,amsmath,amstext,amsbsy,amsopn,amsthm,cases,listings,color
}
\usepackage{ebezier,eepic}
\usepackage{color}
\usepackage{multirow}
\usepackage{epstopdf}
\usepackage{graphicx}
\usepackage{pgf,tikz}
\usepackage{mathrsfs}
\usepackage[marginal]{footmisc}
\usepackage{enumitem}
\usepackage[titletoc]{appendix}
\usepackage{booktabs}
\usepackage{url}
\usepackage{tikz-cd}
\usepackage{subfigure}
\usepackage{mathrsfs}
\usetikzlibrary{arrows}

\allowdisplaybreaks[1]
\usepackage{aligned-overset}

\definecolor{uuuuuu}{rgb}{0.27,0.27,0.27}
\definecolor{sqsqsq}{rgb}{0.1255,0.1255,0.1255}

\newtheorem{definition}{Definition} [section]
\newtheorem{theorem}[definition]{Theorem}
\newtheorem{lemma}[definition]{Lemma}

\newtheorem{claim}[definition]{Claim}
\newtheorem{problem}[definition]{Problem}
\newtheorem{observation}[definition]{Observation}

\newenvironment{pf}{\noindent{\bf Proof}}

\begin{document}
\title{\bf\Large Cancellative hypergraphs and Steiner triple systems}

\date{\today}

\author{
Xizhi Liu
\thanks{Department of Mathematics, Statistics, and Computer Science, University of Illinois, Chicago, IL, 60607 USA. Email: xliu246@uic.edu.
Research partially supported by NSF awards DMS-1763317 and DMS-1952767.}
}
\maketitle
%\footnote{footnote}
%%%%%%%%%%%%%%%%%%%%%%%%%%%%%%%%%%%%%%%%%%%%%%%%%
\begin{abstract}
A triple system is cancellative if it does not contain three distinct sets $A,B,C$ such that
the symmetric difference of $A$ and $B$ is contained in $C$.
We show that every cancellative triple system $\mathcal{H}$ that satisfies certain
inequality between the sizes of $\mathcal{H}$ and its shadow
must be structurally close to the balanced blowup of some Steiner triple system.
Our result contains a stability theorem for cancellative triple systems due to Keevash and Mubayi as a special case.
It also implies that the boundary of the feasible region of cancellative triple systems has
infinitely many local maxima, thus giving the first example showing this phenomenon.

\medskip

\textbf{Keywords}: hypergraph Tur\'{a}n problem, stability, cancellative triple system, Steiner triple system, feasible region.
\end{abstract}
%%%%%%%%%%%%%%%%%%%%%%%%%%%%%%%%%%%%%%%%%%%%%%%%%

\section{Introduction}\label{SEC:Introduction}
Let $r\ge 2$ and $\mathcal{F}$ be a family of $r$-graphs.
An $r$-graph is {\em $\mathcal{F}$-free} if it does not contain any member of $\mathcal{F}$ as a subgraph.
The {\em Tur\'{a}n number} ${\rm ex}(n,\mathcal{F})$ of $\mathcal{F}$ is the maximum size of an $\mathcal{F}$-free $r$-graph on $n$ vertices,
and the {\em Tur\'{a}n density} of $\mathcal{F}$ is $\pi(\mathcal{F}) = \lim_{n\to \infty}{\rm ex}(n,\mathcal{F})/\binom{n}{r}$.
A family $\mathcal{F}$ is {\em nondegenerate} if $\pi(\mathcal{F}) > 0$.
It is one of the central problems in extremal combinatorics to determine ${\rm ex}(n,\mathcal{F})$ for various families $\mathcal{F}$.

Much is known about ${\rm ex}(n,\mathcal{F})$ when $r=2$ and one the most famous results in this regard is Tur\'{a}n's theorem \cite{TU41},
which states that for $\ell \ge 2$ the Tur\'{a}n number ${\rm ex}(n,K_{\ell+1})$ is uniquely achieved by
$T(n,\ell)$ which is the $\ell$-partite graph on $n$ vertices with the maximum number of edges.
However, for $r\ge 3$ determining ${\rm ex}(n,\mathcal{F})$, even $\pi(\mathcal{F})$, is notoriously hard in general.
Compared to the case $r=2$, very little is known about ${\rm ex}(n,\mathcal{F})$ for $r\ge 3$,
and we refer the reader to \cite{KE11} for results before 2011.

For every integer $r\ge 2$ let $\mathcal{T}_{r}$ be the family of $r$-graphs with at most $2r-1$ vertices
and three edges $A,B,C$ such that $A\triangle B \subset C$.
An $r$-graph $\mathcal{H}$ is {\em cancellative} if it has the property that $A\cup B = A\cup C$ implies $B = C$,
where $A,B,C$ are edges in $\mathcal{H}$.
It is easy to see that an $r$-graph is cancellative if and only if it is $\mathcal{T}_{r}$-free.

Let $\ell\ge r \ge 2$ be integers and let $V_{1} \cup \cdots \cup V_{\ell}$ be a partition of $[n]$
with each $V_{i}$ of size either $\lfloor n/\ell \rfloor$ or $\lceil n/\ell \rceil$.
The {\em generalized Tur\'{a}n graph} $T_{r}(n,\ell)$ is the collection of all $r$-subsets of $[n]$ that have
at most one vertex in each $V_{i}$. Let $t_{r}(n,\ell) = |T_{r}(n,\ell)| \sim \binom{\ell}{r}(n/\ell)^r$.

In the 1960's, Katona raised the problem of determining the maximum size of a cancellative $3$-graph on $n$ vertices and
conjectured that the maximum size is achieved by $T_{3}(n,3)$.
Katona's conjecture was later proved by Bollob\'{a}s \cite{BO74},
and it is perhaps the first exact result for a nondegenerate hypergraph family.

\begin{theorem}[Bollob\'{a}s \cite{BO74}]\label{thm-Bollobas-cancel-3}
A cancellative $3$-graph on $n$ vertices has size at most $t_{3}(n,3)$, with equality only for $T_{3}(n,3)$.
\end{theorem}

Bollob\'{a}s conjectured that a similar result holds for all $r\ge 4$.
Sidorenko \cite{SI87} proved it for $r=4$,
but Shearer \cite{SH96} gave a construction showing
that Bollob\'{a}s' conjecture is false for all $r\ge 10$.
The number $\pi(\mathcal{T}_{r})$ is still unknown for all $r\ge 5$.

%%%%%%%%%%%%%%%%%%%%%%%%%%%%%%%%%
\subsection{Stability}\label{SUBSEC:intro-stability}
Keevash and Mubayi \cite{KM04} considered the stability property of $\mathcal{T}_{3}$ and
proved the following theorem.

\begin{theorem}[Keevash-Mubayi \cite{KM04}]\label{THM:KM-stability-cancel-3}
For every $\delta > 0$ there exists $\epsilon > 0$ and $n_0$ such that the following holds for all $n\ge n_0$.
Every cancellative $3$-graph $\mathcal{H}$ with $n$ vertices and at least $(1-\epsilon)t_{3}(n,3)$ edges
can be transformed into a subgraph of $T_{3}(n,3)$ by removing at most $\delta n^3$ edges.
\end{theorem}

{\bf Remarks.}
\begin{itemize}
\item
Keevash and Mubayi actually proved a slightly stronger result which shows that
the statement above holds even for $F_5$-free $3$-graphs, where $F_5 \in \mathcal{T}_{3}$
is the $3$-graph on $5$ vertices with edge set $\{123,124,345\}$.
\item
The author of the present paper gave another short proof in \cite{LIU19} which shows that a linear dependency
between $\delta$ and $\epsilon$ is sufficient for  Theorem~\ref{THM:KM-stability-cancel-3}.
\item
Pikhurko \cite{PI08} proved a similar (and stronger) stability theorem for cancellative $4$-graphs.
Both Theorem~\ref{THM:KM-stability-cancel-3} and Pikhurko's result were further strengthened in \cite{LMR2},
in which Andr\'{a}sfai-Erd\H{o}s-S\'{o}s-type results \cite{AES74} for $\mathcal{T}_{3}$ and $\mathcal{T}_{4}$ were proved.
\end{itemize}

For an $r$-graph $\mathcal{H}$ the {\em shadow} $\partial\mathcal{H}$ of $\mathcal{H}$ is
\[
\partial\mathcal{H} = \left\{ A \in \binom{V(\mathcal{H})}{r-1}: \text{$\exists B\in \mathcal{H}$ such that $A\subset B$} \right\}.
\]
In \cite{LM19A}, several inequalities concerning $\mathcal{H}$ and $\partial\mathcal{H}$ were proved for cancellative $3$-graphs $\mathcal{H}$.

\begin{theorem}[\cite{LM19A}]\label{THM:LM-cancel-3-feasible-region}
Let $n\ge 1$ be an integer and $x\in [0,1]$ be a real number.
Suppose that $\mathcal{H}$ is a cancellative $3$-graph on $n$ vertices with $|\partial\mathcal{H}| = xn^2/2$.
Then
\begin{align}\label{equ:fesible-region-cancel-3-left}
|\mathcal{H}| \le \left(\frac{x}{6}\right)^{3/2}n^3,
\end{align}
and
\begin{align}\label{equ:fesible-region-cancel-3-right}
|\mathcal{H}| \le \frac{x(1-x)}{6}n^3 + 3n^2.
\end{align}
\end{theorem}

We will focus on inequality $(\ref{equ:fesible-region-cancel-3-right})$ in the present paper,
and show that it is closely related to the Steiner triple systems,
which is another classical topic that has been intensively studied in Combinatorics.

For a positive integer $k$ a $k$-vertex {\em Steiner triple system} (${\rm STS}$) is
a $3$-graph on $k$ vertices such that every pair of vertices is contained in exactly one edge.
It is known that a $k$-vertex ${\rm STS}$ exists if and only if $k \in 6\mathbb{N} + \{1,3\}$ (e.g. see \cite{RW03}),
where $6\mathbb{N} + \{1,3\}$ is the set of integers that congruent to $1$ or $3$ modulo $6$.

Let ${\rm STS}(k)$ denote the family of all Steiner triple systems on $k$ vertices up to isomorphism.
In particular, ${\rm STS}(3)$ contains only the $3$-graph $K_3^3$,
${\rm STS}(6)$ contains only the Fano plane (see Figure~\ref{fig:projective-plane-STS-7}),
and ${\rm STS}(9)$ contains only the affine plane of order $3$ (see Figure~\ref{fig:affine-plane-STS-9}).
For $k \in 6\mathbb{N} + \{1,3\}$ let $s_k = |{\rm STS}(k)|$.
Then $s_{3} = s_{7} = s_{9}=1$, $s_{13} = 2$ (see \cite{CO13}), $s_{15} = 80$ (see \cite{HS55}),
and Keevash (see \cite{KE18} and also \cite{WR73,EG81,FA81}) proved that $s_{k} = \left(k/e^2+o(k)\right)^{k^2/6}$.

\begin{figure}[htbp]
\centering
\subfigure{
\begin{minipage}[t]{0.45\linewidth}
\centering
\begin{tikzpicture}[xscale=2.5,yscale=2.5]
\node (a) at (-1,0) {};
\fill (a) circle (0.04);
\node (b) at (0,0) {};
\fill (b) circle (0.04);
\node (c) at (1,0) {};
\fill (c) circle (0.04);
\node (d) at (0.5,0.866025) {};
\fill (d) circle (0.04);
\node (e) at (0,1.732051) {};
\fill (e) circle (0.04);
\node (f) at (-0.5,0.866025) {};
\fill (f) circle (0.04);
\node (g) at (0,0.577350) {};
\fill (g) circle (0.04);
\node (h) at (0,-0.3) {};

\draw[line width=0.8pt] (0,0.577350) circle [radius = 0.577350];

\draw[line width=0.8pt] (-1,0) -- (1,0);
\draw[line width=0.8pt] (-1,0) -- (0,1.732051);
\draw[line width=0.8pt] (-1,0) -- (0.5,0.866025);
\draw[line width=0.8pt] (1,0) -- (0,1.732051);
\draw[line width=0.8pt] (1,0) -- (-0.5,0.866025);
\draw[line width=0.8pt] (0,0) -- (0,1.732051);
\end{tikzpicture}
\caption{The Fano plane.}
\label{fig:projective-plane-STS-7}
\end{minipage}
}
\centering
\subfigure{
\begin{minipage}[t]{0.45\linewidth}
\centering
\begin{tikzpicture}[xscale=1.5,yscale=1.5]
\node (a) at (-1,0) {};
\fill (a) circle (0.06);
\node (b) at (0,0) {};
\fill (b) circle (0.06);
\node (c) at (1,0) {};
\fill (c) circle (0.06);
\node (d) at (1,1) {};
\fill (d) circle (0.06);
\node (e) at (0,1) {};
\fill (e) circle (0.06);
\node (f) at (-1,1) {};
\fill (f) circle (0.06);
\node (g) at (-1,2) {};
\fill (g) circle (0.06);
\node (h) at (0,2) {};
\fill (h) circle (0.06);
\node (i) at (1,2) {};
\fill (i) circle (0.06);

\draw[line width=0.8pt] (-1-0.2,0) -- (1+0.2,0);
\draw[line width=0.8pt] (-1,0-0.2) -- (-1,2+0.2);
\draw[line width=0.8pt] (1,2+0.2) -- (1,0-0.2);
\draw[line width=0.8pt] (1+0.2,2) -- (-1-0.2,2);
\draw[line width=0.8pt] (0,0-0.2) -- (0,2+0.2);
\draw[line width=0.8pt] (-1-0.2,1) -- (1+0.2,1);

\draw[line width=0.8pt] (-1-0.2,2-0.2) to [out = 45, in = 225] (-1,2) to [out = 45, in = 135] (1.3,2.3)
                         to [out = 315, in = 45] (1,1) to [out = 225, in = 45] (0,0) to [out = 225, in = 45] (0-0.2,0-0.2);
\draw[line width=0.8pt] (1-0.2,2+0.2) to [out = 315, in = 135] (1,2) to [out = 315, in = 45] (1.3,-0.3)
                         to [out = 225, in = 315] (0,0) to [out = 135, in = 315] (-1,1) to [out = 135, in = 315] (-1-0.2,1+0.2);
\draw[line width=0.8pt] (1+0.2,0+0.2) to [out = 225, in = 45] (1,0) to [out = 225, in = 315] (-1.3,-0.3)
                         to [out = 135, in = 225] (-1,1) to [out = 45, in = 225] (0,2) to [out = 45, in = 225] (0+0.2,2+0.2);
\draw[line width=0.8pt] (-1+0.2,0-0.2) to [out = 135, in = 315] (-1,0) to [out = 135, in = 225] (-1.3,2.3)
                         to [out = 45, in = 135] (0,2) to [out = 315, in = 135] (1,1) to [out = 315, in = 135] (1+0.2,1-0.2);
\end{tikzpicture}
\caption{The affine plane of order $3$.}
\label{fig:affine-plane-STS-9}
\end{minipage}
}
\end{figure}

\begin{definition}\label{DFN:blowup}
An $r$-graph $\mathcal{H}$ is a blowup of another $r$-graph $\mathcal{G}$ if there exists a
map $\psi\colon V(\mathcal{H}) \to V(\mathcal{G})$ so that $\psi(E) \in \mathcal{G}$ if and only if $E\in \mathcal{H}$,
and we say $\mathcal{H}$ is $\mathcal{G}$-colorable
if there exists a map $\phi\colon V(\mathcal{H}) \to V(\mathcal{G})$ so that $\phi(E)\in \mathcal{G}$ for all $E\in \mathcal{H}$.
\end{definition}

Note that $\mathcal{H}$ is $\mathcal{G}$-colorable if and only if $\mathcal{H}$ occurs as a subgraph in some blowup of $\mathcal{G}$.

The following easy observation relates cancellative $3$-graphs to the Steiner triple systems.

\begin{observation}\label{OBS:blowup-Steiner-is-cancellative}
Suppose that $\mathcal{H}$ is a blowup of a Steiner triple system.
Then $\mathcal{H}$ is cancellative.
Moreover, if $\mathcal{H}$ is a $3$-graph on $n$ vertices that
is a balanced blowup of a Steiner triple system on $k$ vertices,
then $|\partial\mathcal{H}| \sim \frac{k-1}{k}\frac{n^2}{2}$ and
$|\mathcal{H}|\sim \frac{1}{6}\frac{k-1}{k^2}n^3$.
\end{observation}

In particular, the observation above shows that inequality $(\ref{equ:fesible-region-cancel-3-right})$
is tight up to an additive error term $3n^2$ for all $x$ of the form $\frac{k-1}{k}$, where $k\in 6\mathbb{N} + \{1,3\}$ and $k\ge 3$.
The following result makes the relation between cancellative $3$-graphs and the Steiner triple systems more specific.

For an $r$-graph $\mathcal{H}$ on $n$ vertices define the {\em edge density} of $\mathcal{H}$ as
$d(\mathcal{H})= |\mathcal{H}|/\binom{n}{r}$
and define the {\em shadow density} of $\mathcal{H}$ as $d(\partial\mathcal{H})= |\partial\mathcal{H}|/\binom{n}{r-1}$.

\begin{theorem}\label{THM:mian-cancellative-steiner-stability}
Let $k \in 6\mathbb{N} + \{1,3\}$ and $k\ge 3$.
For every $\delta>0$ there exists $\epsilon>0$ and $n_0$ such that the following holds for all $n\ge n_0$.
Suppose that $\mathcal{H}$ is cancellative $3$-graph on $n$ vertices that satisfies
\begin{align}
\left(d(\partial\mathcal{H}) - \frac{k-1}{k}\right)^2 + \left(d(\mathcal{H}) - \frac{k-1}{k^2}\right)^2  \le \epsilon. \notag
\end{align}
Then $\mathcal{H}$ is $\mathcal{S}$-colorable for some $\mathcal{S}\in {\rm STS}(k)$
after removing at most $\delta n^3$ edges.
\end{theorem}

{\bf Remarks.}
\begin{itemize}
\item
Roughly speaking, Theorem~\ref{THM:mian-cancellative-steiner-stability} says if the shadow density and the edge density
of a cancellative $3$-graph $\mathcal{H}$ are close (in the sense of Euclidean distance in $\mathbb{R}^2$)
to $\frac{k-1}{k}$ and $\frac{k-1}{k^2}$ respectively,
then the structure of $\mathcal{H}$ is close to the balanced blowup of a Steiner triple system on $k$ vertices.
\item
Theorem~\ref{THM:mian-cancellative-steiner-stability} contains Keevash and Mubayi's result (Theorem~\ref{THM:KM-stability-cancel-3})
as a special case (i.e. $k=3$) because by inequalities $(\ref{equ:fesible-region-cancel-3-left})$ and $(\ref{equ:fesible-region-cancel-3-right})$
(also see Figure~\ref{fig:feasible-region-cancellative-3}),
if $d(\mathcal{H})$ is close to $2/9$, then $\partial\mathcal{H}$ must be close to $2/3$.
So by Theorem~\ref{THM:mian-cancellative-steiner-stability}, $\mathcal{H}$ is structurally close to the balanced blowup of $K_{3}^{3}$,
which is $T_{3}(n,3)$.
\item
Our proof shows that the relation $\delta = O(\epsilon^{1/2})$ is sufficient for Theorem~\ref{THM:mian-cancellative-steiner-stability}.
\end{itemize}

A more detailed analysis of the proof of Theorem~\ref{THM:mian-cancellative-steiner-stability}
yields the following exact result.

\begin{theorem}\label{thm-cancellative-exact}
Let $k \in 6\mathbb{N} + \{1,3\}$, $k\ge 3$, and $n$ be a sufficiently large integer.
Suppose that $\mathcal{H}$ is a cancellative $3$-graph on $n$ vertices with $|\partial\mathcal{H}| = t_{2}(n,k)$.
Then $|\mathcal{H}|\le s(n,k)$, where
\begin{align}
s(n,k) = \max\left\{|\mathcal{G}|\colon \text{$\mathcal{G}$ is a blowup of $\mathcal{S}$ for
some $\mathcal{S}\in {\rm STS}(k)$ and $|V(\mathcal{G})| = n$}\right\}. \notag
\end{align}
Moreover, equality holds only if $\mathcal{H}$ is a blowup of $\mathcal{S}$ for some $\mathcal{S}\in {\rm STS}(k)$.
\end{theorem}

%%%%%%%%%%%%%%%%%%%%%%%%%%%%%%%%%%%%%%%%%%%%%
\subsection{Feasible region}\label{SUBSEC:intro-feasible-region}
The classical Kruskal-Katona theorem \cite{KA66,KR63} gives a tight upper bound for $|\mathcal{H}|$ as a function of $|\partial \mathcal{H}|$.
In \cite{LM19A}, the hypergraph Tur\'{a}n problem and the Kruskal-Katona theorem were combined and studied systemically.

The {\em feasible region} $\Omega(\mathcal{F})$ of a family $\mathcal{F}$ of $r$-graphs
is the set of points $(x,y)\in [0,1]^2$ such that there exists a sequence of $\mathcal{F}$-free $r$-graphs
$\left( \mathcal{H}_{k}\right)_{k=1}^{\infty}$ with $\lim_{k \to \infty}v(\mathcal{H}_{k}) = \infty$,
$\lim_{k \to \infty}d(\partial\mathcal{H}_{k}) = x$ and $\lim_{k \to \infty}d(\mathcal{H}_{k}) = y$.
Several general results about the shape of $\Omega(\mathcal{F})$ were proved in \cite{LM19A,LM19,LMR1}.
In particular, $\Omega(\mathcal{F})$ is completely determined by a left-continuous almost everywhere differentiable function
$g(\mathcal{F})\colon {\rm proj}\Omega(\mathcal{F}) \to [0,1]$, which we called the {\em feasible region function} of $\mathcal{F}$,
where
\begin{align}
{\rm proj}\Omega(\mathcal{F}) = \left\{ x : \text{$\exists y \in [0,1]$ such that $(x,y) \in \Omega(\mathcal{F})$} \right\}, \notag
\end{align}
and
\begin{align}
g({\mathcal{F}},x) = \max\left\{y: (x,y) \in \Omega(\mathcal{F}) \right\}, \text{ for all } x \in {\rm proj}\Omega(\mathcal{F}) \notag
\end{align}
(here we abuse notation by writing $g({\mathcal{F}})(x)$ as $g({\mathcal{F}},x)$).

In particular, for $\mathcal{T}_{3}$
inequalities $(\ref{equ:fesible-region-cancel-3-left})$ and $(\ref{equ:fesible-region-cancel-3-right})$ imply that
\begin{align}
g(\mathcal{T}_{3},x)\le \min\{x^{3/2}/\sqrt{6},x(1-x)\} \quad{\rm for}\quad x\in[0,1]. \notag
\end{align}
Moreover, the inequality is tight for all $x\in [0,1/2]$ and all $x\in \left\{\frac{k-1}{k}\colon k\in 6\mathbb{N}+\{1,3\}\right\}$
(note that  when $x= \frac{k-1}{k}$ the equality are achieved by the balanced blowup of Steiner triple systems on $k$ vertices).

It was proved in \cite{LM19,LMR1} that the function $g(\mathcal{F})$ can have an arbitrary (finite) number of global maxima in general.
However, it is not clear that whether the function $g(\mathcal{F})$ can have a local maximum.
In particular, the following question was posed in \cite{LM19A}.

\begin{problem}[\cite{LM19A}]\label{prob-steiner-points-maximum}
For every $k \in 6\mathbb{N}+\{1,3\}$ with $k\ge 7$,
is the point $((k-1)/k,(k-1)/k^2)$ a local maximum of $g(\mathcal{T}_{3})$?
\end{problem}

\begin{figure}[htbp]
\centering
\begin{tikzpicture}[xscale=5,yscale=5]
\draw [->] (0,0)--(1.1,0);
\draw [->] (0,0)--(0,0.6);
\draw (0,0.5)--(1,0.5);
\draw (1,0)--(1,0.5);
\draw [line width=1pt,dash pattern=on 1pt off 1.2pt,domain=0:2/3] plot(\x,{2/9});
\draw [line width=1pt,dash pattern=on 1pt off 1.2pt] (2/3,0) -- (2/3,2/9);
\draw [line width=1pt,dash pattern=on 1pt off 1.2pt] (6/7,0) -- (6/7,6/49);
\draw [line width=1pt,dash pattern=on 1pt off 1.2pt] (8/9,0) -- (8/9,8/81);
\draw [line width=1pt,dash pattern=on 1pt off 1.2pt] (0,6/49) -- (6/7,6/49);
\draw [line width=1pt,dash pattern=on 1pt off 1.2pt] (0,8/81) -- (8/9,8/81);
\draw[line width=1pt,color=sqsqsq,fill=sqsqsq,fill opacity=0.25]
(2/3,2/9)--(0.6630236470626231,0.2204032031146704)
--(0.6585443512860356,0.21817345951646344)--(0.654065055509448,0.21595128623168966)
--(0.6495857597328606,0.21373670913866868)--(0.645106463956273,0.21152975438293659)
--(0.6406271681796856,0.20933044838187664)--(0.636147872403098,0.2071388178294629)
--(0.6316685766265105,0.20495488970112097)--(0.6271892808499231,0.2027786912587088)
--(0.6227099850733355,0.2006102500556217)--(0.6182306892967481,0.19844959394202646)
--(0.6137513935201605,0.19629675107022665)--(0.609272097743573,0.19415174990016584)
--(0.6047928019669855,0.19201461920507196)--(0.600313506190398,0.1898853880772468)
--(0.5958342104138105,0.1877640859340081)--(0.591354914637223,0.1856507425237863)
--(0.5868756188606356,0.18354538793238254)--(0.582396323084048,0.18144805258939525)
--(0.5779170273074605,0.17935876727481742)--(0.573437731530873,0.1772775631258134)
--(0.5689584357542855,0.17520447164368083)--(0.5644791399776979,0.17313952470100336)
--(0.5599998442011105,0.17108275454900254)--(0.5555205484245229,0.16903419382509416)
--(0.5510412526479355,0.1669938755606583)--(0.5465619568713479,0.164961833189029)
--(0.5420826610947604,0.16293810055371402)--(0.537603365318173,0.16092271191685042)
--(0.5331240695415854,0.15891570196790752)--(0.528644773764998,0.1569171058326439)
--(0.5241654779884104,0.1549269590823307)--(0.519686182211823,0.1529452977432499)
--(0.5152068864352354,0.15097215830647748)--(0.5107275906586479,0.14900757773796589)
--(0.5062482948820604,0.14705159348893285)--(0.5017689991054729,0.14510424350657278)
--(0.4972897033288854,0.1431655662451021)--(0.4928104075522979,0.14123560067715077)
--(0.48833111177571037,0.13931438630551737)--(0.48385181599912286,0.1374019631752988)
--(0.4793725202225354,0.13549837188641337)--(0.4748932244459478,0.1336036536065323)
--(0.4704139286693603,0.1317178500844364)--(0.46593463289277287,0.12984100366381837)
--(0.46145533711618536,0.1279731572975465)--(0.45697604133959785,0.12611435456241368)
--(0.45249674556301034,0.12426463967439019)--(0.44801744978642283,0.12242405750440374)
--(0.4435381540098353,0.12059265359467064)--(0.43905885823324786,0.11877047417560284)
--(0.43457956245666024,0.11695756618331694)--(0.43010026668007284,0.11515397727777338)
--(0.4256209709034853,0.1133597558615753)--(0.4211416751268978,0.11157495109945828)
--(0.4166623793503103,0.1097996129385029)--(0.4121830835737228,0.10803379212910673)
--(0.4077037877971353,0.10627754024675189)--(0.4032244920205478,0.10453090971460692)
--(0.39874519624396026,0.10279395382700543)--(0.3942659004673728,0.10106672677384512)
--(0.3897866046907853,0.09934928366595529)--(0.38530730891419773,0.09764168056148004)
--(0.3808280131376103,0.09594397449333411)--(0.3763487173610227,0.09425622349778398)
--(0.37186942158443526,0.0925784866442176)--(0.36739012580784774,0.0909108240661645)
--(0.36291083003126023,0.0892532969936367)--(0.3584315342546727,0.08760596778686161)
--(0.35395223847808527,0.08596889997148711)--(0.34947294270149776,0.0843421582753402)
--(0.3449936469249102,0.08272580866683046)--(0.34051435114832274,0.08111991839509318)
--(0.33603505537173517,0.07952455603197399)--(0.3315557595951477,0.07793979151596783)
--(0.3270764638185602,0.07636569619822708)--(0.3225971680419727,0.0748023428907691)
--(0.3181178722653852,0.07324980591701827)--(0.31363857648879767,0.07170816116483066)
--(0.3091592807122102,0.07017748614216043)--(0.30467998493562265,0.06865786003553928)
--(0.3002006891590352,0.06714936377155435)--(0.2957213933824477,0.0656520800815259)
--(0.2912420976058602,0.06416609356960143)--(0.28676280182927266,0.06269149078450184)
--(0.28228350605268515,0.06122836029517568)--(0.27780421027609764,0.05977679277063996)
--(0.27332491449951013,0.058336881064308725)--(0.2688456187229227,0.05690872030314095)
--(0.26436632294633516,0.05549240798196618)--(0.25988702716974765,0.05408804406338191)
--(0.25540773139316014,0.05269573108365302)--(0.25092843561657263,0.05131557426508493)
--(0.24644913983998512,0.04994768163538859)--(0.2419698440633976,0.04859216415460634)
--(0.23749054828681013,0.04724913585022642)--(0.23301125251022262,0.045918713961177304)
--(0.2285319567336351,0.04460101909146803)--(0.2240526609570476,0.04329617537431961)
--(0.21957336518046008,0.042004310647729165)--(0.21509406940387255,0.04072555664250949)
--(0.2106147736272851,0.03946004918396895)--(0.20613547785069758,0.03820792840853079)
--(0.20165618207411007,0.03696933899674451)--(0.19717688629752256,0.035744430424320926)
--(0.19269759052093502,0.03453335723302326)--(0.18821829474434756,0.03333627932348077)
--(0.18373899896776005,0.03215336227225981)--(0.17925970319117254,0.03098477767583987)
--(0.17478040741458506,0.02983070352450204)--(0.17030111163799752,0.028691324609560667)
--(0.16582181586141004,0.027566832967862585)--(0.16134252008482253,0.02645742836805515)
--(0.15686322430823502,0.025363318843811742)--(0.15238392853164753,0.024284721280009886)
--(0.14790463275506,0.023221862058822252)--(0.1434253369784725,0.0221749777738319)
--(0.138946041201885,0.02114431602166764)--(0.1344667454252975,0.020130136282327293)
--(0.12998744964871,0.019132710901390154)--(0.1255081538721225,0.01815232618980085)
--(0.12102885809553499,0.017189283659967112)--(0.11654956231894749,0.01624390142069605)
--(0.11207026654235998,0.015316515758221273)--(0.10759097076577247,0.014407482936513701)
--(0.10311167498918497,0.013517181257608283)--(0.09863237921259746,0.01264601343232816)
--(0.09415308343600996,0.01179440932425909)--(0.08967378765942245,0.010962829146125673)
--(0.08519449188283494,0.01015176720926047)--(0.08071519610624744,0.009361756355687817)
--(0.07623590032965993,0.008593373241477962)--(0.07175660455307244,0.007847244693951734)
--(0.06727730877648493,0.00712405544089494)--(0.06279801299989743,0.006424557617874988)
--(0.05831871722330992,0.0057495826171624865)--(0.053839421446722414,0.005100056076844675)
--(0.04936012567013491,0.0044770171693905856)--(0.04488082989354741,0.003881643919684239)
--(0.040401534116959896,0.0033152872187856555)--(0.03592223834037239,0.0027795178016977455)
--(0.03144294256378489,0.0022761933405199843)--(0.026963646787197385,0.0018075583218622703)
--(0.02248435101060988,0.0013764007788228095)--(0.018005055234022373,0.0009863159625249872)
--(0.01352575945743487,0.000642194917737317)--(0.009046463680847366,0.00035127127796605074)
--(0.00456716790425986,0.00012600704929554806)
--(0,0)--(1,0)
--(0.99,0.0099)--(0.98,0.0196)--(0.97,0.0291)--(0.96,0.0384)--(0.95,0.0475)
--(0.94,0.0564)--(0.93,0.0651)--(0.92,0.0736)--(0.91,0.0819)--(0.9,0.09)--(0.89,0.0979)
--(0.88,0.1056)--(0.87,0.1131)--(0.86,0.1204)--(0.85,0.1275)--(0.84,0.1344)--(0.83,0.1411)
--(0.82,0.1476)--(0.81,0.1539)--(0.8,0.16)--(0.79,0.1659)--(0.78,0.1716)--(0.77,0.1771)
--(0.76,0.1824)--(0.75,0.1875)--(0.74,0.1924)--(0.73,0.1971)--(0.72,0.2016)--(0.71,0.2059)
--(0.7,0.21)--(0.69,0.2139)--(0.68,0.2176)--(0.67,0.2211);
\begin{scriptsize}
\draw [fill=uuuuuu] (2/3,2/9) circle (0.2pt);
\draw [fill=uuuuuu] (6/7,6/49) circle (0.2pt);
\draw [fill=uuuuuu] (8/9,8/81) circle (0.2pt);
%\draw[color=uuuuuu] (2/3+0.12,2/9+0.07) node {$\left(\frac{2}{3},\frac{2}{9}\right)$};
\draw [fill=uuuuuu] (2/3,0) circle (0.2pt);
\draw[color=uuuuuu] (2/3,0-0.07) node {$\frac{2}{3}$};
\draw [fill=uuuuuu] (6/7,0) circle (0.2pt);
\draw[color=uuuuuu] (6/7-0.01,0-0.07) node {$\frac{6}{7}$};
\draw [fill=uuuuuu] (8/9,0) circle (0.2pt);
\draw[color=uuuuuu] (8/9+0.01,0-0.07) node {$\frac{8}{9}$};
\draw [fill=uuuuuu] (1,0) circle (0.2pt);
\draw[color=uuuuuu] (1,0-0.07) node {$1$};
\draw [fill=uuuuuu] (0,0) circle (0.2pt);
\draw[color=uuuuuu] (0-0.05,0-0.05) node {$0$};
\draw [fill=uuuuuu] (0,2/9) circle (0.2pt);
\draw[color=uuuuuu] (0-0.08,2/9) node {$2/9$};
\draw [fill=uuuuuu] (0,6/49) circle (0.2pt);
\draw[color=uuuuuu] (0-0.08,6/49+0.02) node {$6/49$};
\draw [fill=uuuuuu] (0,8/81) circle (0.2pt);
\draw[color=uuuuuu] (0-0.08,8/81-0.02) node {$9/81$};
\draw [fill=uuuuuu] (0,1/2) circle (0.2pt);
\draw[color=uuuuuu] (0-0.08,1/2) node {$1/2$};
\draw[color=uuuuuu] (1+0.1,0-0.07) node {$x$};
\draw[color=uuuuuu] (0-0.08,1/2+0.1) node {$y$};
\end{scriptsize}
\end{tikzpicture}
\caption{$g(\mathcal{T}_{3},x)\le \min\{x^{3/2}/\sqrt{6},x(1-x)\}$ for $x\in[0,1]$.}
\label{fig:feasible-region-cancellative-3}
\end{figure}
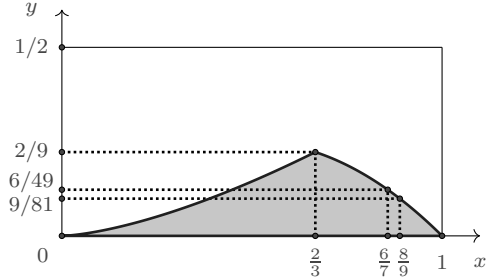

As an application of Theorem~\ref{THM:mian-cancellative-steiner-stability}
we give an affirmative answer to Problem~\ref{prob-steiner-points-maximum},
and this is the first example showing that the feasible region function can have a local maximum.

\begin{theorem}\label{THM:local-max-cancelltive-3-fesible-region}
Let $k \in 6\mathbb{N} + \{1,3\}$ and $k\ge 3$ be fixed.
Then there exists an absolute constant $c>0$ such that for
every constant $\epsilon \le c$ there exists another constant $\delta = \delta(\epsilon)>0$
so that
\begin{align}
g(\mathcal{T}_{3},(k-1)/k-\epsilon) \le \frac{k-1}{k^2} - \delta
\quad{\rm and}\quad
g(\mathcal{T}_{3},(k-1)/k+\epsilon) \le \frac{k-1}{k^2} - \delta. \notag
\end{align}
In particular, the point $((k-1)/k,(k-1)/k^2)$ is a local maximum of $g(\mathcal{T}_{3})$.
\end{theorem}

{\bf Remark.}
Our proof shows that a linear dependency between $\delta$ and $\epsilon$ is sufficient for
Theorem~\ref{THM:local-max-cancelltive-3-fesible-region}.

The remainder of this paper is organized as follows.
In Section~\ref{SEC:proof-of-stability} we prove Theorem~\ref{THM:mian-cancellative-steiner-stability}.
In Section~\ref{SEC:feasible-region} we prove Theorem~\ref{THM:local-max-cancelltive-3-fesible-region}.
In Section~\ref{SEC:exact-result} we prove Theorem~\ref{thm-cancellative-exact}.
We include some remarks in Section~\ref{SEC:remarks}.

%%%%%%%%%%%%%%%%%%%%%%%%%%%%%%%%%%%%%%%%%%%%%%%%
\section{Proof of Theorem~\ref{THM:mian-cancellative-steiner-stability}}\label{SEC:proof-of-stability}
In this section we prove Theorem~\ref{THM:mian-cancellative-steiner-stability}.
Let us present some preliminary results first.

\subsection{Preliminaries}
For an $r$-graph $\mathcal{H}$ and a set $S \subset V(\mathcal{H})$
denote by $\mathcal{H}[S]$ the induced subgraph of $\mathcal{H}$ on $S$.
For two disjoint set $S,T \subset V(\mathcal{H})$ denote by $\mathcal{H}[S,T]$
the collection of edges in $\mathcal{H}$ that have nonempty intersection with both $S$ and $T$.

For a vertex $v\in V(\mathcal{H})$ the {\em neighborhood} of $v$ is
\begin{align}
N_{\mathcal{H}}(v) = \left\{u\in V(\mathcal{H})\setminus \{v\}\colon
                       \exists A\in \mathcal{H} \text{ such that }\{u,v\}\subset A\right\}, \notag
\end{align}
and the {\em link} $L_{\mathcal{H}}(v)$ of $v$ in $\mathcal{H}$ is
\begin{align}
L_{\mathcal{H}}(v) = \left\{A\in \partial\mathcal{H}\colon A\cup\{v\} \in \mathcal{H}\right\}. \notag
\end{align}
The {\em degree} of $v$ is $d_{\mathcal{H}}(v) = |L_{\mathcal{H}}(v)|$.
Denote by $\Delta(\mathcal{H})$ and $\delta(\mathcal{H})$ the {\em maximum degree} and the {\em minimum degree} of $\mathcal{H}$, respectively.
We will omit the subscript $\mathcal{H}$ if it is clear from the context.

For a pair of vertices $u,v\in V(\mathcal{H})$ the {\em neighborhood} of $uv$ (we use $uv$ as a shorthand for $\{u,v\}$) is
\begin{align}
N_{\mathcal{H}}(uv) = \left\{w\in V(\mathcal{H})\setminus \{u,v\}\colon
                       \exists A\in \mathcal{H} \text{ such that }\{u,v,w\}\subset A\right\}, \notag
\end{align}
and the size of $N_{\mathcal{H}}(uv)$ is called the {\em codegree} of $uv$.
Denote by $\Delta_{2}(\mathcal{H})$ and $\delta_{2}(\mathcal{H})$ the {\em maximum codegree}
and the {\em minimum codegree} of $\mathcal{H}$, respectively.

For a graph $G$ the {\em clique number} $\omega(G)$ of $G$
is the largest integer $\omega$ such that $K_{\omega} \subset G$.

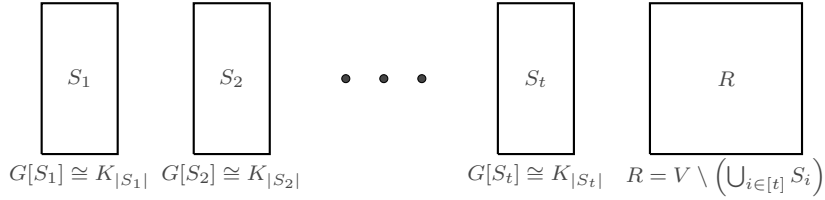
\begin{figure}[htbp]
\centering
\begin{tikzpicture}[xscale=1,yscale=1]
\draw[line width=0.8pt]
(0,0)--(1,0)--(1,2)--(0,2)--(0,0);
\draw[line width=0.8pt]
(2,0)--(3,0)--(3,2)--(2,2)--(2,0);
\draw[line width=0.8pt]
(6,0)--(7,0)--(7,2)--(6,2)--(6,0);
\draw[line width=0.8pt]
(8,0)--(10,0)--(10,2)--(8,2)--(8,0);
%\draw[line width=0.5pt]
%(6,-0.45)--(6.1,-0.65)--(9.9,-0.65)--(10,-0.45);
\begin{scriptsize}
\draw[color=uuuuuu] (0.5,1) node {$S_1$};
\draw[color=uuuuuu] (0.5,0-0.3) node {$G[S_1] \cong K_{|S_1|}$};
\draw[color=uuuuuu] (2.5,1) node {$S_2$};
\draw[color=uuuuuu] (2.5,0-0.3) node {$G[S_2] \cong K_{|S_2|}$};
\draw[color=uuuuuu] (6.5,1) node {$S_t$};
\draw[color=uuuuuu] (6.5,0-0.3) node {$G[S_t] \cong K_{|S_t|}$};
\draw[color=uuuuuu] (9,1) node {$R$};
\draw[color=uuuuuu] (9,0-0.3) node {$R = V\setminus\left(\bigcup_{i\in[t]}S_i\right)$};
%\draw[color=uuuuuu] (8,0-0.9) node {$G_{t-1}, \mathcal{H}_{t-1}$};
\draw [fill=uuuuuu] (4,1) circle (1.5pt);
\draw [fill=uuuuuu] (4.5,1) circle (1.5pt);
\draw [fill=uuuuuu] (5,1) circle (1.5pt);
\end{scriptsize}
\end{tikzpicture}
\caption{A clique expansion of graph $G$.}
\label{FIG:clique-expansion}
\end{figure}

\begin{definition}[Clique expansion]
Let $t\ge 1$, $\kappa \ge 1$ be positive integers and $G$ be a graph on the set $V$.
\begin{itemize}
\item[(a)]
A $(t+1)$-tuple $(S_1,\ldots, S_{t},R)$, where $S_1, \ldots, S_t$ are pairwise disjoint subsets of $V(G)$ and
$R = V\setminus\left(S_1\cup\cdots \cup S_t\right)$,
is a clique expansion of $G$ if $G[S_i]$ is a complete graph for every $i\in[t]$ (see Figure~\ref{FIG:clique-expansion}).
\item[(b)]
A clique expansion $(S_1,\ldots, S_{t},R)$ is maximal if
the size of $S_i$ equals the clique number of the induced subgraph of $G$ on $V\setminus\left(S_1\cup\cdots \cup S_{i-1}\right)$
for $i \in [t]$.
\item[(c)]
We say a clique expansion $(S_1,\ldots, S_{t},R)$ has a threshold $\kappa$ if $|S_i| \ge \kappa$ for every $i\in[t]$
but $\omega\left(G[R]\right)< \kappa$.
\end{itemize}
\end{definition}

The following observation is immediate from the definition.

\begin{observation}\label{OBS:max-clique-expansion}
Suppose that $(S_1,\ldots, S_{t},R)$ is a maximal clique expansion of $G$.
Then
\begin{itemize}
\item[(a)]
$|S_1| \ge \cdots \ge |S_t|$, and
\item[(b)]
every vertex in $V\setminus\left(S_1\cup\cdots \cup S_{i}\right)$ is adjacent to at most $|S_i| - 1$ vertices in $S_i$ for $i\in[t]$.
%In particular, ???
%\begin{align}
%E_i = \sum_{j=1}^{i}e_{j} \le \sum_{j=1}^{i}\left(|T_i|(\omega_i - 1) + \binom{\omega_i}{2}\right)
%\le \sum_{j=1}^{i}(\omega_i-1)n = \left(W_i-i\right)n. \notag
%\end{align}
\end{itemize}
\end{observation}

We will need the following result due to Andr\'{a}sfai, Erd\H{o}s, and S\'{o}s \cite{AES74}.

\begin{theorem}[Andr\'{a}sfai-Erd\H{o}s-S\'{o}s \cite{AES74}]\label{THM:AES74}
Let $k \ge 2$ and $n\ge 1$ be positive integers.
Then every $K_{k+1}$-free graph on $n$ vertices with minimum degree greater than $\frac{3k-4}{3k-1}n$ is a $k$-partite graph.
\end{theorem}

It will be convenient sometimes to consider $\mathcal{H}$ and $\partial\mathcal{H}$ separately.

\begin{definition}[Cancellative pair]\label{DFN:cancellative-pair}
Let $G$ be a graph on $V$ and $\mathcal{H}$ be a $3$-graph on the same vertex set $V$.
We say the pair $(G,\mathcal{H})$ is cancellative if $\partial\mathcal{H} \subset G$ and
it does not contain three distinct sets $A,B\in \mathcal{H}$ and $C\in G\cup\mathcal{H}$ such that $A\triangle B\subset C$.
We call $V$ the vertex set of the pair $(G,\mathcal{H})$.
\end{definition}

Let $\mathcal{H}$ be a cancellative $3$-graph and $U\subset V(\mathcal{H})$.
Then it is easy to see that the pair $\left((\partial\mathcal{H})[U], \mathcal{H}[U]\right)$ is cancellative
(note that $(\partial\mathcal{H})[U]$ and $\partial(\mathcal{H}[U])$ are not necessarily the same).

\begin{observation}\label{OBS:cancellative-pair-co-neighbor-is-independent}
Suppose that $(G,\mathcal{H})$ is a cancellative pair. Then
$N_{\mathcal{H}}(uv)$ is an independent set in $G$ for every $uv\in \partial\mathcal{H}$.
\end{observation}

The following results concerning cancellative pairs were proved in \cite{LM19A}.

\begin{theorem}[\cite{LM19A}]\label{thm-calcel-3-induction-shadow}
Let $m\ge 1$ be an integer and $(G,\mathcal{H})$ be a cancellative pair on a set $V$ of size $m$.
Suppose that $|G|=xm^2/2$ for some real number $x \in [0,1]$. Then
\begin{align}
|\mathcal{H}|\le \frac{x(1-x)}{6}m^3+3m^2. \notag
\end{align}
\end{theorem}

\begin{lemma}[\cite{LM19A}]\label{LEMMA:cancellative-3-graphs-clique}
Let $(G,\mathcal{H})$ be a cancellative pair on a set $V$.
Suppose that $G[S]$ is a complete graph for some set $S\subset V$.
Then
\[
\sum_{v\in S}d_{\mathcal{H}}(v) \le |\partial\mathcal{H}|.
\]
\end{lemma}

Lemma~\ref{LEMMA:cancellative-3-graphs-clique} yields the following result.

\begin{lemma}\label{LEMMA:cancellative-clique-expansion}
Let $t\ge 1$ be a positive integer and $(G,\mathcal{H})$ be a cancellative pair on $n$ vertices.
Suppose that $(S_1,\ldots,S_t,R)$ is a clique expansion of $G$.
Then
\begin{align}
|\mathcal{H}| \le |\mathcal{H}[R]| + t |\partial\mathcal{H}|. \notag
\end{align}
\end{lemma}
\begin{proof}[Proof of Lemma~\ref{LEMMA:cancellative-clique-expansion}]
Notice that every edge in $\mathcal{H}$ either contains at least one vertex in $S_1\cup \cdots \cup S_{t}$
or is completely contained in $R$.
So,
\begin{align}
|\mathcal{H}|
\le |\mathcal{H}[R]| + \sum_{i\in[t]}\sum_{v\in S_i}d_{\mathcal{H}}(v). \notag
\end{align}
For every $i\in[t]$
since $G[S_i]$ is a complete graph,
it follows from Lemma~\ref{LEMMA:cancellative-3-graphs-clique} that
$\sum_{v\in S_i}d_{\mathcal{H}}(v) \le |\partial\mathcal{H}|$.
Therefore, $|\mathcal{H}|
\le |\mathcal{H}[R]| + t|\partial\mathcal{H}|$.
\end{proof}

%\begin{lemma}[\cite{LM19A}]\label{LEMMA:analysis-inequality}
%Suppose that $x\in [2/3,1]$ and $x'\in [0,1]$ are real numbers,
%$m\ge \omega \ge 1/(1-x)$ and
%$0\le e \le (m-\omega)(\omega-1)+\binom{\omega}{2}$ are integers,
%and $x' = (xm^2-2e)/(m-\omega)^2$ holds.
%Then
%\begin{align}
%\frac{(1-x)x}{6}m^3 + 3m^2 \ge \frac{(1-x')x'}{6}(m-\omega)^3 + 3(m-\omega)^2 + \frac{x}{2}m^2. \notag
%\end{align}
%\end{lemma}
%%%%%%%%%%%%%%%%%%%%%%%%%%%%%%%%%%%%%%%%%%%%%%%%%%%%%%%
\subsection{Proof of Theorem~\ref{THM:mian-cancellative-steiner-stability}}\label{SUBSEC:proof-stability-Steiner}
In this section we prove the following statement that implies Theorem~\ref{THM:mian-cancellative-steiner-stability}.

\begin{theorem}\label{THM:cancel-3-stability-steiner-real-statement}
Let $k\in 6\mathbb{N}+\{1,3\}$ and $k\ge 3$.
For every $\delta>0$ there exists an $\epsilon>0$ and $n_0$ such that the following holds for all $n\ge n_0$.
Suppose that $\mathcal{H}$ is a cancellative $3$-graph on $n$ vertices with
\begin{align}\label{equ:shaodw-and-edge-sizes-assumption}
|\partial\mathcal{H}|\ge (1-\epsilon)\frac{k-1}{2k}n^2
\quad{\rm and}\quad
|\mathcal{H}|\ge (1-\epsilon)\frac{k-1}{6k^2}n^3.
\end{align}
Then, $\mathcal{H}$ is $\mathcal{S}$-colorable for some $\mathcal{S}\in {\rm STS}(k)$ after removing at most $\delta n^3$ edges.
\end{theorem}

{\bf Remarks.}
\begin{itemize}
\item[(a)]
Our proof shows that $\delta = 20000k^6\epsilon^{1/2}$ is sufficient for Theorem~\ref{THM:cancel-3-stability-steiner-real-statement}.
\item[(b)]
It is easy to see from inequality $(\ref{equ:fesible-region-cancel-3-right})$ (also see Figure~\ref{fig:feasible-region-cancellative-3})
that if $\mathcal{H}$ satisfies $(\ref{equ:shaodw-and-edge-sizes-assumption})$,
then the Euclidean distance between $\left(d(\partial\mathcal{H}),d(\mathcal{H})\right)$ and
$\left(\frac{k-1}{k},\frac{k-1}{k^2}\right)$ is bounded by some constant $\zeta = \zeta(\epsilon)$ that is linear in $\epsilon$,
and vice versa (we omit the detailed calculations here).
\end{itemize}

The technical parts of the proof of Theorem~\ref{THM:mian-cancellative-steiner-stability}
are contained in proofs of Lemma~\ref{LEMMA:cancellative-3-graph-stability-decompose-vertex-set}
and Lemma~\ref{LEMMA:k-partite-cancellative-3-graph-stability}.
In Lemma~\ref{LEMMA:cancellative-3-graph-stability-decompose-vertex-set} we will
show that every $3$-graph $\mathcal{H}$ that satisfies assumptions in Theorem~\ref{THM:cancel-3-stability-steiner-real-statement}
contains a small set $U$ of vertices such that the induced subgraph of
$\partial\mathcal{H}$ on $V(\mathcal{H})\setminus U$ is $k$-partite.
In Lemma~\ref{LEMMA:k-partite-cancellative-3-graph-stability} we will show that
if a cancellative pair $(G,\mathcal{H})$ satisfies similar assumptions in Theorem~\ref{THM:cancel-3-stability-steiner-real-statement}
and $G$ is $k$-partite,
then $\mathcal{H}$ is $\mathcal{S}$-colorable for some $\mathcal{S}\in {\rm STS}(k)$ after removing very few number of edges.

\begin{lemma}\label{LEMMA:cancellative-3-graph-stability-decompose-vertex-set}
Let $k\in 6\mathbb{N}+\{1,3\}$ and $k\ge 3$.
There exists an absolute constant $c_{1} = c_{1}(k)>0$ such that
for every constant $\epsilon$ satisfying $0\le \epsilon \le c_1$
there exists $n_0 = n_{1}(k,\epsilon)$ such that the following holds for all $n\ge n_1$.
Suppose that $\mathcal{H}$ is a cancellative $3$-graph on $n$ vertices
that satisfies $(\ref{equ:shaodw-and-edge-sizes-assumption})$.
Then there exists a set $U\subset V(\mathcal{H})$ of size at most $130\epsilon k^4 n$ such that the induced subgraph
of $\partial\mathcal{H}$ on $V(\mathcal{H})\setminus U$ is $k$-partite.
\end{lemma}

\begin{lemma}\label{LEMMA:k-partite-cancellative-3-graph-stability}
Let $k\in 6\mathbb{N}+\{1,3\}$ and $k\ge 3$.
There exists an absolute constant $c_{2} = c_{2}(k)>0$ such that
for every constant $\epsilon'$ satisfying $0\le \epsilon' \le c_2$
there exists $n_2 = n_{2}(k,\epsilon')$ such that the following holds for all $n\ge n_2$.
Suppose that $(G,\mathcal{H})$ is a cancellative pair on a set $V$ of size $n$, $G$ is $k$-partite,
\begin{align}\label{equ:cancellative-pair-inequalities}
|G|\ge (1-\epsilon')\frac{k-1}{2k}n^2
\quad{\rm and}\quad
|\mathcal{H}|\ge (1-\epsilon')\frac{k-1}{6k^2}n^3.
\end{align}
Then, $\mathcal{H}$ is $\mathcal{S}$-colorable for some $\mathcal{S}\in {\rm STS}(k)$
after removing at most $600(\epsilon')^{1/2}k^3 n^3$ edges.
\end{lemma}

First let us show that Lemma~\ref{LEMMA:cancellative-3-graph-stability-decompose-vertex-set}
and Lemma~\ref{LEMMA:k-partite-cancellative-3-graph-stability} imply Theorem~\ref{THM:cancel-3-stability-steiner-real-statement}.

\begin{proof}[Proof of Theorem~\ref{THM:cancel-3-stability-steiner-real-statement} assuming Lemmas~\ref{LEMMA:cancellative-3-graph-stability-decompose-vertex-set} and~\ref{LEMMA:k-partite-cancellative-3-graph-stability}]
Let $\epsilon> 0$ be a sufficiently small constant such that $\epsilon' = 800\epsilon k^6 n$ satisfies $\epsilon' \le c_2$.
Let $\delta = 20000\epsilon^{1/2}k^6$.
Suppose that $\mathcal{H}$ is a cancellative $3$-graph on a set $V$ of $n$ vertices and $\mathcal{H}$ satisfies assumptions in
Theorem~\ref{THM:cancel-3-stability-steiner-real-statement}.
By Lemma~\ref{LEMMA:cancellative-3-graph-stability-decompose-vertex-set},
there exists a set $U\subset V$ of size at most $130\epsilon k^4 n$ such that the induced subgraph of $\partial\mathcal{H}$
on $V\setminus U$ is $k$-partite.
Let $G' = (\partial\mathcal{H})[V\setminus U]$ and $\mathcal{H}' = \mathcal{H}[V\setminus U]$.
Then it is easy to see that $(G',\mathcal{H}')$ is a cancellative pair, and moreover,
\begin{align}
|G'|
\ge |\partial\mathcal{H}| - |U|\cdot n
\ge (1-\epsilon)\frac{k-1}{2k}n^2 - 130\epsilon k^4 n^2
\ge (1-\epsilon')\frac{k-1}{2k}n^2, \notag
\end{align}
and
\begin{align}
|\mathcal{H}'|
\ge |\mathcal{H}| - |U|\cdot n^2
\ge (1-\epsilon)\frac{k-1}{6k^2}n^3 - 130\epsilon k^4 n^3
\ge (1-\epsilon')\frac{k-1}{6k^2}n^3. \notag
\end{align}
Therefore, by Lemma~\ref{LEMMA:k-partite-cancellative-3-graph-stability},
$\mathcal{H}'$ contains subgraph $\mathcal{H}''$ of size at least $|\mathcal{H}'| - 600(\epsilon')^{1/2}k^3 n^3$
such that $\mathcal{H}''$ is $\mathcal{S}$-colorable for some $\mathcal{S}\in {\rm STS}(k)$.
Note that $|\mathcal{H}| - |\mathcal{H}''| \le 600(\epsilon')^{1/2}k^3 n^3 + 130\epsilon k^4 n^3 < 20000\epsilon^{1/2}k^6 n^3$.
This completes the proof of Theorem~\ref{THM:cancel-3-stability-steiner-real-statement}.
\end{proof}

%%%%%%%%%%%%%%%%%%%%%%%%%%%%%%%%%%%
\subsection{Proof of Lemma~\ref{LEMMA:cancellative-3-graph-stability-decompose-vertex-set}}\label{SUBSEC:proof-first-lemma}
\begin{proof}[Proof of Lemma~\ref{LEMMA:cancellative-3-graph-stability-decompose-vertex-set}]
Fix $k\in 6\mathbb{N}+\{1,3\}$ and $k\ge 3$.
Let $\epsilon > 0$ be a sufficiently small constant and $n$ be a sufficiently large integer.
Let $\mathcal{H}$ be a cancellative $3$-graph on $n$ vertices and assume that $\mathcal{H}$ satisfies assumptions in Lemma~\ref{LEMMA:cancellative-3-graph-stability-decompose-vertex-set}.

\begin{claim}\label{CLAIM:shadow-H-upper-bound}
We have $|\partial\mathcal{H}|\le \left(\frac{k-1}{2k}+\epsilon\right)n^2$.
\end{claim}
\begin{proof}[Proof of Claim~\ref{CLAIM:shadow-H-upper-bound}]
Let $x = 2|\partial\mathcal{H}|/n^2$ and suppose to the contrary that
$x \ge \left((k-1)/k+2\epsilon\right)$.
Then it follows from Theorem~\ref{thm-calcel-3-induction-shadow} that
\begin{align}
|\mathcal{H}|
\le \frac{(1-x)x}{6}n^3 + 3n^2
& \le \frac{1}{6}\left(\frac{k-1}{k}+2\epsilon\right)\left(\frac{1}{k}-2\epsilon\right)n^3 + 3n^2 \notag\\
& \le \frac{k-1}{6k^2}\left(1-\frac{2k(k-2)}{k-1}\epsilon\right) n^3 + 3n^2
 < \left(1-\epsilon\right)\frac{k-1}{6k^2}, \notag
\end{align}
which contradicts $(\ref{equ:shaodw-and-edge-sizes-assumption})$.
\end{proof}%CLAIM

\begin{claim}\label{CLAIM:clique-number-shadow-H-upper-bound}
The clique number $\omega\left(\partial\mathcal{H}\right)$ of $\partial\mathcal{H}$
satisfies $\omega\left(\partial\mathcal{H}\right) \le 10k \epsilon n$.
\end{claim}
\begin{proof}[Proof of Claim~\ref{CLAIM:clique-number-shadow-H-upper-bound}]
Suppose to the contrary that there exists a set $S\subset V(\mathcal{H})$ of size
$\lceil 10k\epsilon n\rceil$ such the the induced subgraph of $\partial\mathcal{H}$
on $S$ is complete.
To keep the calculations simple, let us assume that $10k\epsilon n$ is an integer.
Let $a = 10k \epsilon$, $R = V(\mathcal{H})\setminus S$, and $e=|\partial\mathcal{H}|$.
Let $e_{s}$ be the number of edges in $\partial\mathcal{H}$ that have at least one vertex in $S$,
and set $x' = (xn^2-e_{s})/(n-an)^2$.
Notice that $e_{s} \le an^2$.

It follows from Lemma~\ref{LEMMA:cancellative-3-graphs-clique} and Theorem~\ref{thm-calcel-3-induction-shadow} that
\begin{align}\label{ineq-stability-3-H-up-bound-x'}
|\mathcal{H}|
& \le |\mathcal{H}[T]|+\sum_{v\in S}d(v) \notag\\
& \le \frac{(1-x')x'}{6}(1-a)^3n^3+3(1-a)^2n^2+ e \notag\\
& = \frac{((1-a)^2n^2-2(e-e_s))(e-e_s)}{3(1-a)n}+3(1-a)^2n^2+ e \notag\\
& = \frac{-2e_s^2+\left(4e-(1-a)^2n^2 \right)e_s + (1-a)^2 n^2 e-2e^2 }{3(1-a)n}+3(1-a)^2n^2+ e.
\end{align}
Since $-2e_s^2+\left(4e-(1-a)^2n^2 \right)e_s$ is increasing in $e_s$ when $e_s \le e-(1-a)^2n^2/4$ and
\begin{align}
e-\frac{(1-a)^2n^2}{4}
& > (1-\epsilon)\frac{(k-1)n^2}{2k}-\frac{(1-a)^2n^2}{4} \notag\\
& \ge (1-\epsilon)\frac{n^2}{3}-\frac{(1-a)^2n^2}{4}
> an^2, \notag
\end{align}
we may substitute $e_{s} = an^2$ into (\ref{ineq-stability-3-H-up-bound-x'}) and obtain
\begin{align}\label{ineq-stability-3-H-up-bound-2}
|\mathcal{H}|
\le \frac{-2e^2 +\left( (1+a)^2 n^2 +3(1-a)n \right) e - (1+a^2)an^4 }{3(1-a)n}+3(1-a)^2n^2.
\end{align}
Since $-2e^2 +\left( (1+a)^2 n^2 +3(1-a)n \right) e$ is decreasing in $e$ when $e\ge (1+a)^2n^2/4 + 3(1-a)n/4$
and
\begin{align}
\frac{(1+a)^2n^2}{4}+\frac{3(1-a)n}{4}
< \frac{n^2}{4}+\frac{n^2}{100}
< (1-\epsilon)\frac{n^2}{3}
< (1-\epsilon)\frac{(k-1)n^2}{2k}, \notag
\end{align}
we may substitute $e=(1-\epsilon)(k-1)n^2/(2k)$ into (\ref{ineq-stability-3-H-up-bound-2}) and obtain
\begin{align}
|\mathcal{H}|
& < \frac{((1-\epsilon)(k-1) -2ka)(1+ka^2+(k-1)\epsilon)}{6(1-a)k^2} n^3 + \left( (1-\epsilon)\frac{(k-1)}{2k}+ 3(1-a)^2\right) n^2 \notag\\
& \le \frac{\left( (k-1)(1-a) - ka \right) \left( 1+ k \epsilon \right)}{6(1-a)k^2}  n^3 + 4n^2 \notag\\
& \le \frac{k-1}{6k^2} n^3 + \frac{\epsilon}{6} n^3 - \frac{a}{6k}n^3 + 4n^2
 < \frac{k-1}{6k^2} n^3 - \epsilon n^3, \notag
\end{align}
which contradicts $(\ref{equ:shaodw-and-edge-sizes-assumption})$.
\end{proof}%CLAIM

\begin{claim}\label{claim-stability-sum-Si-up-bound}
Suppose that $(S_1,\ldots,S_{t},R)$ is a
maximal clique expansion of $\partial\mathcal{H}$ with $|S_t| \ge k+1$ for some positive integer $t$.
Then $\sum_{i\in[t]}|S_i| < 30k^2 \epsilon n$.
\end{claim}
\begin{proof}[Proof of Claim~\ref{claim-stability-sum-Si-up-bound}]
Suppose to the contrary that there exist some positive integer $t$ and
an maximal clique expansion $(S_1,\ldots,S_{t},R)$ of $\partial\mathcal{H}$ with $|S_t| \ge k+1$
such that $\sum_{i\in[t]}|S_i| \ge 30k^2 \epsilon n$.
Let $\Sigma_t = \sum_{i\in[t]}|S_i|$.

Let $\beta = 20k^2 \epsilon$.
By Claim~\ref{CLAIM:clique-number-shadow-H-upper-bound}, $|S_i| < 10k\epsilon n$ for $i\in[t]$.
So there exists $t'\le t$ such that $\beta n - 10k\epsilon n \le W_{t'} \le \beta n+ 10k\epsilon n < 30k^2 \epsilon n$.
Without loss of generality we may assume that $W_t = \lceil \beta n \rceil$
(since otherwise we may replace $W_t$ by $W_{t'}$, and the exact value of $\beta$ is not crucial in the proof
as long as $20k^2 \le \beta \le 30k^2$).
To keep the calculations simple, let us assume that $\beta n$ is an integer.

Let $E_t$ denote the number of edges in $\partial\mathcal{H}$ that have at least one vertex in $S_{1}\cup \cdots \cup S_t$
and  let $x' = 2(e-E_t)/(n-\Sigma_t)^2$.
Notice from Observation~\ref{OBS:max-clique-expansion}~$(b)$ that
\begin{align}
E_t
\le \sum_{i\in[t]}\left(|S_i|-1\right)n
= \left(\Sigma_{t}-t\right)n. \notag
\end{align}
It follows from Theorem~\ref{thm-calcel-3-induction-shadow} and Lemma~\ref{LEMMA:cancellative-clique-expansion} that
\begin{align}\label{inequ-stability-H-W-up-bound-1}
|\mathcal{H}|
& \le \frac{x'(1-x')}{6} (n-\Sigma_t)^3 + 3(n-\Sigma_t)^2 + t e \notag\\
& = \frac{-2E_t^2 +\left(4e-(n-\Sigma_t)^2 \right)E_t+ (n-\Sigma_t)^2 e-2e^2}{3(n-\Sigma_t)}+ 3(n-\Sigma_t)^2 + t e.
\end{align}
Similar to the proof of Claim~\ref{CLAIM:clique-number-shadow-H-upper-bound},
we may substitute $E_t = \left(\Sigma_t - t\right)n$ into (\ref{inequ-stability-H-W-up-bound-1}) and obtain
\begin{align}\label{inequ-stability-H-W-up-bound-2}
|\mathcal{H}|
\le \frac{-2n^2 t^2 + \left(n(n+\Sigma_t)^2-(n+3\Sigma_t)e \right)t - (e-\Sigma_t n)(2e-n^2-\Sigma_t^2)}{3(n-\Sigma_t)} + 3(n-\Sigma_t)^2.
\end{align}
Since $|S_i|\ge k+1$ for $i\in[t]$, we have $t \le W_t/(k+1)$.
Since $-2n^2 t^2 + \left(n(n+\Sigma_t)^2-(n+3\Sigma_t)e \right)t $ is increasing in $t$ when
\begin{align}
t \le  \left( n(n+\Sigma_t)^2-(n+3\Sigma_t)e \right)/({4n^2}) \notag
\end{align}
and $\left( n(n+\Sigma_t)^2-(n+3\Sigma_t)e \right)/({4n^2}) \ge \Sigma_t/(k+1)$,
we may substitute $t = \Sigma_t/(k+1)$ into (\ref{inequ-stability-H-W-up-bound-2}) and obtain
\begin{align}\label{inequ-stability-H-W-up-bound-3}
|\mathcal{H}|
& \le \frac{(k+1)\left(-2(k+1)e^2+\left( (k+1)n^2 + (2k+1)\Sigma_tn +(k-2)\Sigma_t^2 \right)e\right)}{3(k+1)^2(n-\Sigma_t)} \notag\\
& \quad - \frac{\left( (k+1)(n^2+ \Sigma_t^2)-2\Sigma_tn  \right)k\Sigma_tn}{3(k+1)^2(n-\Sigma_t)} + 3(n-\Sigma_t)^2.
\end{align}
Since $-2(k+1)^2e^2+(k+1)\left( (k+1)n^2 + (2k+1)\Sigma_tn +(k-2)\Sigma_t^2 \right)e$ is decreasing in $e$ when
\begin{align}
e\ge \frac{(k+1)n^2 + (2k+1)\Sigma_tn +(k-2)\Sigma_t^2}{4(k+1)} \notag
\end{align}
and
\begin{align}
(1-\epsilon)\frac{k-1}{2k}n^2 \ge \frac{(k+1)n^2 + (2k+1)\Sigma_tn +(k-2)\Sigma_t^2}{4(k+1)}, \notag
\end{align}
we may substitute $e = (1-\epsilon)(k-1)n^2/(2k)$ into (\ref{inequ-stability-H-W-up-bound-3}) and obtain
\begin{align}
|\mathcal{H}|
& \le (1-\epsilon)\frac{k-1}{6k^2} n^3 - \frac{(k+1)^2\Sigma_tn^3 - k(k^3+2k^2-k+2)\Sigma_t^2n^2+2k^3(k+1)\Sigma_t^3n }{6k^2(k+1)^2(n-\Sigma_t)} \notag\\
& \quad +\epsilon n^3 + 3(n-\Sigma_t)^2 \notag\\
& < (1-\epsilon)\frac{k-1}{6k^2} n^3 - \frac{(k+1)^2\Sigma_tn^3}{12k^2(k+1)^2n} + +\epsilon n^3 + 3(n-\Sigma_t)^2 \notag\\
& < (1-\epsilon)\frac{k-1}{6k^2} n^3 - \frac{\beta n^3}{12k^2} + +\epsilon n^3 + 3(n-\Sigma_t)^2
 < (1-\epsilon)\frac{k-1}{6k^2} n^3 \notag
\end{align}
contradicting $(\ref{equ:shaodw-and-edge-sizes-assumption})$. Here we used $\beta = 20k^2 \epsilon$.
\end{proof}%CLAIM

Now let $(S_1,\ldots, S_t, R)$ be a maximal clique expansion of $\partial\mathcal{H}$ with threshold $k+1$ for some positive integer $t$.
Let $\tilde{n} = |R|$ and $G = (\partial\mathcal{H})[R]$.
Notice that by the definition of threshold, $G$ is $K_{k+1}$-free.
It follows from Claim~\ref{claim-stability-sum-Si-up-bound} that
$\tilde{n} = n-\sum_{i\in[t]}|S_i| \le n - 30k^2\epsilon n$ and
\begin{align}\label{equ:shadow-on-R-lower-bound}
|G|
> |\partial\mathcal{H}| - 30k^2\epsilon n \cdot n
\ge \left(1-\epsilon\right)\frac{k-1}{2k}n^2 - 30k^2\epsilon n^2
> \frac{k-1}{2k}n^2 - 31k^2\epsilon n^2.
\end{align}
Define
\begin{align}
Z(G)
= \left\{v\in R\colon d_{G}(v) \le \frac{3k-4}{3k-1}\tilde{n}+ 100k^4\epsilon \tilde{n}\right\}. \notag
\end{align}

\begin{claim}\label{CLAIM:Z-upper-bound}
We have $|Z(G)| < 100k^4\epsilon \tilde{n}$.
\end{claim}
\begin{proof}[Proof of Claim~\ref{CLAIM:Z-upper-bound}]
Let $z = |Z(G)|$ and suppose to the contrary that $z \ge 100k^4\epsilon \tilde{n}$.
To keep the calculations simple, let us assume that $100k^4\epsilon\tilde{n}$ is an integer.
We may assume that $z = 100k^4\epsilon\tilde{n}$ since otherwise we replace $Z(G)$ by a subset of size $100k^4\epsilon\tilde{n}$.
Let $R' = R\setminus Z(G)$.
Since $G[R']$ is $K_{k+1}$-free, by Tur\'{a}n's theorem $|G[R']| \le \frac{k-1}{2k}\left(\tilde{n}-z\right)^2$.
Therefore,
\begin{align}
|G|
& \le \frac{k-1}{2k}\left(\tilde{n}-z\right)^2 + \left(\frac{3k-4}{3k-1}\tilde{n}+100k^4\epsilon \tilde{n}\right)z \notag\\
& \le  \frac{k-1}{2k} \tilde{n}^2 - \frac{1}{3k^2-k}z\tilde{n} + z^2 + 100k^4\epsilon \tilde{n} z \notag\\
& < \frac{k-1}{2k} \tilde{n}^2 - \frac{100k^4\epsilon}{3k^2-k}\tilde{n}^2 + 2\left(100k^4\epsilon\tilde{n}\right)^2
 < \frac{k-1}{2k}\tilde{n}^2 - 31k^2\epsilon\tilde{n}^2, \notag
\end{align}
which contradicts $(\ref{equ:shadow-on-R-lower-bound})$.
\end{proof}%CLAIM

Let $U = S_{1}\cup \cdots \cup S_t \cup Z(G)$.
Then by Claims~\ref{claim-stability-sum-Si-up-bound} and~\ref{CLAIM:Z-upper-bound},
$|U| \le 30k^2\epsilon n + 100k^4\epsilon \tilde{n} < 130k^4\epsilon n$.
On the other hand, by $(\ref{equ:shadow-on-R-lower-bound})$ and Claim~\ref{CLAIM:Z-upper-bound}
the induced subgraph of $G$ on $V(\mathcal{H})\setminus U$ has minimum degree at least
\begin{align}
\frac{3k-4}{3k-1}\tilde{n}+ 100k^4\epsilon \tilde{n} - |Z(G)|\tilde{n}
> \frac{3k-4}{3k-1}\tilde{n}. \notag
\end{align}
So by Theorem~\ref{THM:AES74}, the induced subgraph $G[R\setminus Z(G)]$ is $k$-partite.
In other words, the induced subgraph of $\partial\mathcal{H}$ on $V(\mathcal{H})\setminus U$ is $k$-partite.
This completes the proof of Lemma~\ref{LEMMA:cancellative-3-graph-stability-decompose-vertex-set}.
\end{proof}%LEMMA

%%%%%%%%%%%%%%%%%%%%%%%%%%%%%%%
\subsection{Proof of Lemma~\ref{LEMMA:k-partite-cancellative-3-graph-stability}}\label{SUBSEC:proof-second-lemma}
We prove Lemma~\ref{LEMMA:k-partite-cancellative-3-graph-stability} in this section.
The following observation about the Steiner triple systems will be helpful to understand the proof
(starting from Claim~\ref{claim-cancel-3-L(v)-disjoint-bipartite}) of
Lemma~\ref{LEMMA:k-partite-cancellative-3-graph-stability}.

\begin{observation}\label{OBS:blowup-Steiner-triple-system}
Suppose that $\mathcal{H}$ is $\mathcal{S}$-colorable for some $\mathcal{S}\in {\rm STS}(k)$. Then
for every $v\in V(\mathcal{H})$ the link $L_{\mathcal{H}}(v)$ consists of $(k-1)/2$ pairwise vertices disjoint complete graphs.
\end{observation}

\begin{proof}[Proof of Lemma~\ref{LEMMA:k-partite-cancellative-3-graph-stability}]
Fix $k\in 6\mathbb{N}+\{1,3\}$ and $k\ge 3$.
Let $\epsilon' > 0$ be a sufficiently small constant and $n$ be a sufficiently large integer.
Let $(G,\mathcal{H})$ be a cancellative pair on $n$ vertices that satisfies assumptions in Lemma~\ref{LEMMA:k-partite-cancellative-3-graph-stability}.
Let $V = V(G) = V(\mathcal{H})$ and
suppose that $V = V_{1}\cup \cdots \cup V_{k}$ is a partition such that every edge in $G$ contains
at most one vertex from each $V_i$.
Let $\widehat{G}$ denote the complete $k$-partite graph with $k$-parts $V_1,\ldots,V_{k}$.
Let $M_{G} = \widehat{G}\setminus G$ and call members in $M_{G}$ missing edges of $G$.
Notice that
\begin{align}\label{equ:missing-shadows-upper-bound}
|M_{G}|
= \sum_{1\le i < j \le k}|V_i||V_j| - |G|
\le \frac{k-1}{2k}n^2 - \left(1-\epsilon'\right)\frac{k-1}{2k}n^2
= \frac{k-1}{2k}\epsilon' n^2
< \frac{\epsilon' n^2}{2}.
\end{align}

The following claim can be proved easily using the following inequality (see \cite{LMR2})
\begin{align}
\sum_{1\le i < j \le k}x_ix_j
+ \frac{1}{2}\sum_{i\in [k]}\left(x_i-\frac{1}{k}\right)^2
\le \frac{k-1}{2k}, \notag
\end{align}
where $x_1,\ldots,x_k \in[0,1]$ are real numbers satisfying $x_1+\cdots+x_k = 1$.

\begin{claim}\label{claim-Vi-up-low-bound}
We have $\left||V_i| - n/k\right| \le 2(\epsilon')^{1/2}n$ for every $i \in [k]$.
\end{claim}

The next claim gives an upper bound for the maximum codegree of $\mathcal{H}$.

\begin{claim}\label{CLAIM:size-common-neighbors}
We have $\Delta_2(\mathcal{H}) \le n/k + 2(\epsilon')^{1/2} n + k\epsilon' n \le n/k + 3(\epsilon')^{1/2} n$.
\end{claim}
\begin{proof}[Proof of Claim~\ref{CLAIM:size-common-neighbors}]
Suppose to the contrary that there exists $uv\in\partial\mathcal{H}$ with
$N_{\mathcal{H}}(uv) > n/k + 2(\epsilon')^{1/2} n + k\epsilon n$.
Let $V_i' = N_{\mathcal{H}}(uv)\cap V_i$ and $x_i = |V_i'|$ for $i\in[k]$.
By Observation~\ref{OBS:cancellative-pair-co-neighbor-is-independent}, $N_{\mathcal{H}}(uv)$ is independent in $G$.
Therefore, every pair $\{u_i,u_j\}$ with $u_i\in V_i'$, $u_j\in V_j'$, and $i\neq j$, is a member in $M_{G}$.
In particular, $\sum_{1\le i < j \le k}x_ix_j \le |M_{G}|$.
Claim \ref{claim-Vi-up-low-bound} implies that $x_{i}\le n/k + 2(\epsilon')^{1/2} n$ for $i\in[k]$,
which combined with $\sum_{i\in[k]}x_i = |N_{\mathcal{H}}(uv)| \ge n/k + 2(\epsilon')^{1/2} n + k\epsilon' n$ imply that
\begin{align}
\sum_{1\le i < j \le k}x_ix_j
\ge \left(\frac{n}{k}+2(\epsilon')^{1/2} n\right)
 \left(\frac{n}{k}+ 2(\epsilon')^{1/2} n + k\epsilon' n-\left(\frac{n}{k}+2(\epsilon')^{1/2} n\right)\right)
\ge \epsilon' n^2 > |M_{G}| \notag
\end{align}
contradicting $(\ref{equ:missing-shadows-upper-bound})$.
\end{proof}%CLAIM

Define the set of edges in $G$ with small codegree in $\mathcal{H}$ as
\begin{align}
G_{s} = \left\{uv\in G\colon |N_{\mathcal{H}}(uv)| \le \frac{n}{2k}\right\}. \notag
\end{align}

\begin{claim}\label{claim-cancel-3-BE-up-bound}
We have $|G_{s}|< 4k (\epsilon')^{1/2}n^2$.
\end{claim}
\begin{proof}[Proof of Claim \ref{claim-cancel-3-BE-up-bound}]
Suppose to the contrary that $|G_{s}|\ge 4k (\epsilon')^{1/2}n^2$.
Then it follows from $\sum_{uv\in G}|N_{\mathcal{H}}(uv)| \ge 3|\mathcal{H}|$,
Claims~\ref{CLAIM:shadow-H-upper-bound} and~\ref{CLAIM:size-common-neighbors} that
\begin{align}
3|\mathcal{H}|
& \le \frac{n}{2k}|G_{s}| + \left(\frac{n}{k}+3(\epsilon')^{1/2}n\right)\left(|G|-|G_{s}|\right) \notag\\
& \le \left(\frac{n}{k}+3(\epsilon')^{1/2}n\right)|G| - \frac{n}{2k}|G_{s}| \notag\\
& \le \left(\frac{n}{k}+3(\epsilon')^{1/2}n\right)\left(\frac{k-1}{2k}+\epsilon'\right)n^2 - 2(\epsilon')^{1/2}n^3
 \le \frac{k-1}{2k^2} - \frac{(\epsilon')^{1/2}}{4}n^3 \notag
\end{align}
contradicting $(\ref{equ:cancellative-pair-inequalities})$.
\end{proof}%CLAIM

The following claim shows that for every $uv \in G$ most vertices in $N_{\mathcal{H}}(uv)$
will be contained in some $V_i$. It will be used intensively in the remaining part of the proof.

\begin{claim}\label{claim-cancel-3-N(uv)-concenrate}
For every $uv \in G$ and every $i \in [k]$ either
\begin{align}
|N_{\mathcal{H}}(uv)\cap V_i| < \frac{\epsilon' n^2}{|N_{\mathcal{H}}(uv)|}
\quad{\rm or}\quad
|N_{\mathcal{H}}(uv)\cap V_i| > |N_{\mathcal{H}}(uv)|-\frac{\epsilon' n^2}{|N_{\mathcal{H}}(uv)|}. \notag
\end{align}
In particular, if $|N_{\mathcal{H}}(uv)| > (\epsilon' k)^{1/2}n$,
then there exists a unique $i \in [k]$ such that
$|N_{\mathcal{H}}(uv)\cap V_i| > |N_{\mathcal{H}}(uv)|- \epsilon' n^2/|N_{\mathcal{H}}(uv)|$.
\end{claim}
\begin{proof}[Proof of Claim~\ref{claim-cancel-3-N(uv)-concenrate}]
Fix $uv\in G$ and $i \in [k]$.
We may assume that $|N_{\mathcal{H}}(uv)| \ge (2\epsilon')^{1/2}n$ since otherwise we would have
\begin{align}
\frac{\epsilon' n^2}{|N_{\mathcal{H}}(uv)|} \ge |N_{\mathcal{H}}(uv)|-\frac{\epsilon' n^2}{|N_{\mathcal{H}}(uv)|}, \notag
\end{align}
and there is nothing to prove.

Suppose that there exists a set $V_i$ contradicting the assertion of Claim~\ref{claim-cancel-3-N(uv)-concenrate}.
Let $\alpha=|N_{\mathcal{H}}(uv)|$ and $\beta = |N_{\mathcal{H}}(uv)\cap V_i|$.
Then similar to the proof of Claim~\ref{CLAIM:size-common-neighbors} we obtain
\begin{align}
|M_{G}|
\ge \beta(\alpha-\beta)
\ge \frac{\epsilon' n^2}{\alpha}\left(\alpha - \frac{\epsilon' n^2}{\alpha}\right)
= \epsilon' n^2 - \left(\frac{\epsilon' n^2}{\alpha}\right)^2
\ge \frac{\epsilon'}{2}n^2, \notag
\end{align}
which contradicts $(\ref{equ:missing-shadows-upper-bound})$.

Now, suppose that $|N_{\mathcal{H}}(uv)| > (\epsilon' k)^{1/2}n$.
Since
\begin{align}
k \times \frac{\epsilon' n^2}{|N_{\mathcal{H}}(uv)|}
< \frac{\epsilon' kn^2}{(\epsilon' k)^{1/2} n}
= (\epsilon' k)^{1/2} n
< |N_{\mathcal{H}}(uv)|, \notag
\end{align}
there exists $i \in [k]$ such that
$|N_{\mathcal{H}}(uv)\cap V_i| > |N_{\mathcal{H}}(uv)|-\epsilon' n^2/{|N_{\mathcal{H}}(uv)|}$.
Since
\begin{align}
|N_{\mathcal{H}}(uv)|-\frac{\epsilon' n^2}{|N_{\mathcal{H}}(uv)|}
> |N_{\mathcal{H}}(uv)|-\frac{\epsilon' n^2}{(\epsilon' k)^{1/2} n}
= |N_{\mathcal{H}}(uv)|- \frac{(\epsilon')^{1/2}}{k^{1/2}}n
> \frac{|N_{\mathcal{H}}(uv)|}{2}, \notag
\end{align}
such $i$ is unique.
\end{proof}%CLAIM

\begin{claim}\label{CLAIM:cancel-3-max-degree-mathcal-H}
We have $\Delta(\mathcal{H})<\frac{k-1}{2k^2}n^2 + 3(\epsilon')^{1/2} n^2$.
\end{claim}
\begin{proof}[Proof of Claim~\ref{CLAIM:cancel-3-max-degree-mathcal-H}]
Fix $v \in V$ and it suffices to show that $d_{\mathcal{H}}(v)<\frac{k-1}{2k^2}n^2 + 3(\epsilon')^{1/2} n^2$.
Without loss of generality, we may assume that $v\in V_1$.
For every vertex $w \in N_{\mathcal{H}}(v) \subset \bigcup_{i=2}^{k}V_{i}$
let $d_{v}(w)$ denote the degree of $w$ in the link graph $L_{\mathcal{H}}(v)$.
It follows from Claim~\ref{CLAIM:size-common-neighbors} that $d_v(w) = |N_{\mathcal{H}}(vw)| < n/k + 3(\epsilon')^{1/2}n$.
Therefore, by Claim~\ref{claim-Vi-up-low-bound},
\begin{align}
d_{\mathcal{H}}(v)
= |L_{\mathcal{H}}(v)|
= \frac{1}{2} \sum_{w\in \bigcup_{i=2}^{k}V_{i}} d_{v}(w)
& < \frac{1}{2}(k-1)\left(\frac{n}{k}+ 2(\epsilon')^{1/2}n \right) \left(\frac{n}{k} + 3(\epsilon')^{1/2}n\right) \notag\\
& < \frac{k-1}{2k^2}n^2 + 3(\epsilon')^{1/2} n^2. \notag
\end{align}
\end{proof}%CLAIM

Define the set of vertices with small degree in $\mathcal{H}$ as
\begin{align}
V_{s} = \left\{ v\in V\colon d_{\mathcal{H}}(v) < \frac{k-1}{2k^2}n^2 - 18(\epsilon')^{1/2} k n^2 \right\}. \notag
\end{align}

\begin{claim}\label{claim-cancel-3-small-degree-vertices-up-bound}
We have $|V_{s}| < \frac{n}{6k}$.
\end{claim}
\begin{proof}[Proof of Claim~\ref{claim-cancel-3-small-degree-vertices-up-bound}]
Suppose to the contrary that $|V_{s}| \ge \frac{n}{6k}$.
Then by Claim~\ref{CLAIM:cancel-3-max-degree-mathcal-H}, we have
\begin{align}
3|\mathcal{H}|
 = \sum_{v\in V}d_{\mathcal{H}}(v)
& \le \left(\frac{k-1}{2k^2}n^2 + 3(\epsilon')^{1/2} n^2\right)\left(n-\frac{n}{6k}\right)
   + \left(\frac{k-1}{2k^2}n^2 - 18 (\epsilon')^{1/2} k n^2\right) \frac{n}{6k} \notag\\
& \le \left(\frac{k-1}{2k^2} + 3(\epsilon')^{1/2}\right)n^3 - \left(18 (\epsilon')^{1/2} k + 3(\epsilon')^{1/2}\right) \frac{n^3}{6k} \notag\\
& \le \frac{k-1}{2k^2}n^3 - \frac{(\epsilon')^{1/2}}{2k}n^3 \notag
\end{align}
contradicting $(\ref{equ:cancellative-pair-inequalities})$.
\end{proof}%CLAIM

First we show that if $v\in V\setminus V_s$, then the structure of $L_{\mathcal{H}}(v)$ is close to
what we expect to see in a blowup of Steiner triple systems on $k$ vertices.

\begin{claim}\label{claim-cancel-3-L(v)-disjoint-bipartite}
For every $v\in V\setminus V_{s}$ there exists a subgraph of
$L_{\mathcal{H}}(v)$ of size at least $d_{\mathcal{H}}(v) - 243(\epsilon')^{1/2}k^3n^2$
that consists of $(k-1)/2$ pairwise vertex disjoipnt bipartite graphs.
\end{claim}
\begin{proof}[Proof of Claim~\ref{claim-cancel-3-L(v)-disjoint-bipartite}]
Fix $v\in V\setminus V_{s}$ and without loss of generality we may assume that $v\in V_1$.
Note that $L_{\mathcal{H}}(v)$ is a graph on $N_{\mathcal{H}}(v) \subset \bigcup_{i=2}^{k}V_{i}$.
For every $u \in N_{\mathcal{H}}(v)$ let $d_{v}(u)$ denote the degree of $u$ in the link $L_{\mathcal{H}}(v)$.
Define the set of vertices with small degree in $L_{\mathcal{H}}(v)$ as
\begin{align}
N_{s} = \left\{u\in \bigcup_{i=2}^{k}V_{i}\colon d_{v}(u) \le \frac{n}{k} - 240(\epsilon')^{1/2}k^2 n\right\}.   \notag
\end{align}
We claim that $|N_{s}| \le \frac{n}{6k}$ since otherwise we would have
\begin{align}
2|L_{\mathcal{H}}|
& = \sum_{v\in V\setminus V_1}d_{v}(w) \notag\\
& \le \left(\frac{n}{k}+3(\epsilon')^{1/2}n\right)\left(|V\setminus V_1|-|N_{s}|\right)
      + \left(\frac{n}{k} - 240(\epsilon')^{1/2}k^2 n\right) |N_{s}| \notag\\
& \le \left(\frac{n}{k}+3(\epsilon')^{1/2}n\right)\left(\frac{k-1}{k}n+2(\epsilon')^{1/2}n\right)
      - \left(3(\epsilon')^{1/2}n+ 240(\epsilon')^{1/2}kn\right)\frac{n}{6k} \notag\\
& < \frac{k-1}{k}n^2 - 37(\epsilon')^{1/2}kn^2, \notag
\end{align}
which contradicts the assumption that $v\in V\setminus V_{s}$.
Therefore, by Claim~\ref{claim-Vi-up-low-bound}, for every $i \in [2,k]$ we have
\begin{align}\label{inequ-cancel-3-V-minus-Bv-low-bound}
|V_{i}\setminus N_{s}|
> \left(\frac{n}{k}-2(\epsilon')^{1/2}n\right) - \frac{n}{6k}
> \frac{2n}{3k}.
\end{align}
Fix $i \in [2,k]$ and a vertex $u \in V_{i}\setminus N_{s}$.
By Claim \ref{claim-cancel-3-N(uv)-concenrate},
there exists a unique $\psi(u,i) \in [2,k]$, $\psi(u,i)\neq i$, such that
\begin{align}
|N_{\mathcal{H}}(uv) \cap V_{\psi(u,i)}|
 > |N_{\mathcal{H}}(uv)|-\frac{\epsilon' n^2}{|N_{\mathcal{H}}(uv)|}
& = d_{v}(u) - \frac{\epsilon' n^2}{d_v(u)} \notag\\
& > \left(\frac{n}{k} - 240(\epsilon')^{1/2}k^2n\right) - \frac{\epsilon' n^2}{n/(2k)} \notag\\
& > \frac{n}{k} - 241(\epsilon')^{1/2}k^2n. \notag
\end{align}
Since $\psi(u,i) \in [2,k]$ for every $u\in V_{i}\setminus N_{s}$,
it follows from (\ref{inequ-cancel-3-V-minus-Bv-low-bound}) and the Pigeonhole principle that
there exists a set $U_i \subset V_i\setminus B_v$ with $|U_i| > 2n/(3k(k-2))$
such that $\psi(u,i) = \psi(u',i)$ for every pair $u,u'\in U_i$.
We abuse notation by letting $\psi(i) = \psi(u,i)$ for $u\in U_{i}$.

Define the bipartite graph $G_{i,\psi(i)}$ as
\begin{align}
G_{i,\psi(i)} = \left\{ uw \in L_{\mathcal{H}}(v)\colon u\in U_{i}, w \in V_{\psi(i)} \right\}, \notag
\end{align}
and notice that
\begin{align}\label{inequ-cancel-3-Gii'-low-bound}
|G_{i,\psi(i)}| > |U_i| \left(\frac{n}{k} - 241(\epsilon')^{1/2}k^2n\right).
\end{align}
Let
$$V_{\psi(i)}' = \left\{w\in V_{i'}\colon d_{G_{i,\psi(i)}}(w)> |U_i|/2 > n/(3k(k-2))\right\}.$$
Then it follows from
$\sum_{w\in U_{i}} d_{G_{i,\psi(i)}}(w) = |G_{i,\psi(i)}| = \sum_{w\in V_{\psi(i)}} d_{G_{i,\psi(i)}}(w)$ that
\begin{align}
|U_i| \left(\frac{n}{k} - 241(\epsilon')^{1/2}k^2n\right)
< |V_{\psi(i)}'||U_i| + \left(|V_{\psi(i)}|-|V_{\psi(i)}'|\right)\frac{|U_i|}{2}
= \frac{|U_i|}{2}\left(V_{\psi(i)}+V_{\psi(i)}'\right), \notag
\end{align}
which with Claim \ref{claim-Vi-up-low-bound} imply that
\begin{align}
|V_{\psi(i)}'|
& > 2\left(\frac{n}{k} - 241(\epsilon')^{1/2}k^2n\right) - |V_{\psi(i)}'| \notag\\
& > 2\left(\frac{n}{k} - 241(\epsilon')^{1/2}k^2n\right) - \left(\frac{n}{k}+2(\epsilon')^{1/2}n\right)
  > \frac{n}{k} - 242(\epsilon')^{1/2}k^2n. \notag
\end{align}
For every $w \in V_{\psi(i)}'$ since
\begin{align}
|N_{\mathcal{H}}(vw) \cap V_{i}| \ge d_{G_{i,\psi}}(w)> \frac{n}{3k(k-2)}
> \frac{\epsilon' n^2}{|N_{\mathcal{H}}(vw)|}, \notag
\end{align}
it follows from Claim \ref{claim-cancel-3-N(uv)-concenrate} that, in fact,
\begin{align}
|N_{\mathcal{H}}(vw) \cap V_{i}|
> |N_{\mathcal{H}}(vw)| - \frac{\epsilon' n^2}{|N_{\mathcal{H}}(vw)|}
& > |N_{\mathcal{H}}(vw)| - \frac{\epsilon' n^2}{n/(3k(k-2))} \notag\\
& > |N_{\mathcal{H}}(vw)| - 3\epsilon' k^2n. \notag
\end{align}
Consequently, $\psi(w,\psi(i)) = i$ for every $w\in V_{\psi(i)}'$, and we abuse notation by writing it as $\psi^2(i)=i$.

Repeating the argument above we obtain a set $V_{i}' \subset V_{i}$ of size at least
$n/k - 242(\epsilon')^{1/2}k^2n$ such that
$|N_{\mathcal{H}}(vw) \cap V_{i}| \ge |N_{\mathcal{H}}(vw)| - 3\epsilon' k^2n$
for every $w\in V_{i}'$.

View $\psi$ as a map from $[2,k]$ to $[2,k]$.
Then due to $\psi^2(i) = i$ for $i\in[2,k]$
the map $\psi$ defines a perfect matching, namely $\{\{i,\psi(i)\}\colon i\in[2,k]\}$
(note that $\{i,\psi(i)\}$ and $\{\psi(i),\psi(\psi(i))\}$ are the same), on $[2,k]$.

To keep the notations simple, let us assume that $\psi(i) = (k-1)/2 + i$ for $2\le i \le (k+1)/2$.
Then the argument above implies that the number of edges in $L_{\mathcal{H}}(v)$ that are not contained in the union of
the induced bipartite subgraphs $\bigcup_{i=2}^{(k+1)/2}L_{\mathcal{H}}(v)[V_i,V_{\psi(i)}]$ is at most
\begin{align}
& \sum_{i=2}^{(k+1)/2}\left(3\epsilon' k^2n\left(|V_{i}'|+|V_{\psi(i)}'|\right)
+ n\left(|V_{i}\setminus V_{i}'|+|V_{\psi(i)}\setminus V_{\psi(i)}'|\right)\right) \notag\\
& \le \frac{k-1}{2}\left(3\epsilon' k^2n^2
+ 2\left(242(\epsilon')^{1/2}k^2n+2(\epsilon')^{1/2}n\right)n\right)
 \le 243(\epsilon')^{1/2}k^3n^2. \notag
\end{align}
\end{proof}%CLAIM

The proof of Claim~\ref{claim-cancel-3-L(v)-disjoint-bipartite} implies that for every $i\in [k]$ and every $v\in V_{i}\setminus V_{s}$
there is a bijection $\psi_{v}\colon [k]\setminus\{i\} \to [k]\setminus\{i\}$ with $\psi_{v}^{2}(j) = j$ for every $j\in [k]\setminus\{i\}$
such that all but at most $243(\epsilon')^{1/2}k^3n^2$ edges in $L_{\mathcal{H}}(v)$ is contained in
the union of the induced bipartite subgraphs $\bigcup_{j\in [k]\setminus\{i\}}L_{\mathcal{H}}(v)[V_j,V_{\psi_{v}(j)}]$.

Now fix $i\in[k]$ and without loss of generality we may assume that $i=1$.
By the Pigeonhole principle, there exists a set $W_{1}\subset V_1\setminus V_{s}$
of size at least $|V_1\setminus V_{s}|/(k-1)! > {n}/{(2k!)}$
(here $(k-1)!$ is an upper bound for the number of bijections between $[2,k]$ and $[2,k]$)
such that $\psi_{v} \equiv \psi_{v'}$ for every pair $v,v'\in W_1$.
To keep the notations simple, let $\psi_{1}$ be the bijection that satisfies $\psi_{1} \equiv \psi_{v}$ for $v\in W_1$,
and further assume that $\psi_{1}(i) = (k-1)/2+i$ for $2\le i \le (k+1)/2$.

Define an auxiliary bipartite graph $M$
with two parts $P_1 = W_1$ and $P_2 = \bigcup_{i=2}^{(k+1)/2}V_{i}\times V_{\psi_{1}(i)}$
such that a pair $\{v, (u,w)\}$ is an edge in $M$ if and only if $\{u,w\}\in L_{\mathcal{H}}(v)$.
It follows from Claim~\ref{claim-cancel-3-L(v)-disjoint-bipartite} that for every $v\in P_1$ we have
\begin{align}
d_{M}(v)
\ge d_{\mathcal{H}}(v) - 243(\epsilon')^{1/2}k^3n^2
& > \left(\frac{k-1}{2k^2}n^2-18(\epsilon')^{1/2}kn^2\right)- 243(\epsilon')^{1/2}k^3n^2  \notag\\
& > \frac{k-1}{2k^2}n^2 - 270(\epsilon')^{1/2}k^3n^2. \notag
\end{align}
On the other hand, notice from Claim~\ref{claim-Vi-up-low-bound} that
\begin{align}
|P_2|
\le \frac{k-1}{2}\left(\frac{n}{k}+2(\epsilon')^{1/2}n\right)^{2}
< \frac{k-1}{2k^2}n^2 + 2(\epsilon')^{1/2} n^2. \notag
\end{align}
Let $P_2'$ denote the set of vertices in $P_2$ that have degree at least $|P_1|/2$ in $M$.
Then it follows from $\sum_{v\in P_1}d_{M}(v) = |G| = \sum_{e\in P_2}d_{M}(e)$ that
\begin{align}
|P_1|\left(\frac{k-1}{2k^2}n^2 - 270(\epsilon')^{1/2}k^3n^2\right)
\le |P_2'| |P_1| + \left(|P_2|-|P_2'|\right) \frac{P_1}{2}
= \frac{|P_1|}{2}\left(|P_2|+|P_2'|\right), \notag
\end{align}
which implies that
\begin{align}
|P_2'|
& \ge 2\left(\frac{k-1}{2k^2}n^2 - 270(\epsilon')^{1/2}k^3n^2\right) - |P_2| \notag\\
& > 2\left(\frac{k-1}{2k^2}n^2 - 270(\epsilon')^{1/2}k^3n^2\right) - \left(\frac{k-1}{2k^2}n^2 + 2(\epsilon')^{1/2} n^2\right) \notag\\
& > \frac{k-1}{2k^2}n^2 - 541(\epsilon')^{1/2}k^3n^2. \notag
\end{align}
For a pair $(u,w)\in P_2'$ since
\begin{align}
|N_{\mathcal{H}}(uw)\cap V_1|
\ge |N_{\mathcal{H}}(uw)\cap W_1|
\ge d_{M}((u,w))
\ge \frac{|W_1|}{2}
\ge \frac{n}{4k!}, \notag
\end{align}
it follows from Claim~\ref{claim-cancel-3-N(uv)-concenrate} that, in fact,
\begin{align}
|N_{\mathcal{H}}(uw)\cap V_1|
\ge |N_{\mathcal{H}}(uw)| - \frac{\epsilon' n^2}{|N_{\mathcal{H}}(uw)|}
\ge |N_{\mathcal{H}}(uw)| - \frac{\epsilon' n^2}{n/(4k!)}
\ge |N_{\mathcal{H}}(uw)| - 4\epsilon' k! n. \notag
\end{align}
Let
\begin{align}
\mathcal{S} = \left\{\{i,j,\psi_{i}(j)\}\colon i\in[k]\text{ and }j\in[k]\setminus \{i\}\right\} \notag
\end{align}
(every edge in $\mathcal{S}$ appeared six times in the definition above but we only keep one of them).
Notice that $\mathcal{S}$ is a Steiner triple system on $[k]$.
Let $\widehat{\mathcal{S}}$ be the blowup of $\mathcal{S}$ obtained by replacing each vertex $i$ by the set $V_i$
and replacing each edge by a corresponding complete $3$-partite $3$-graph.
Then the argument above implies that one can delete at most
\begin{align}
& k\times \left(|P_2\setminus P_2'|\times n + |P_2'|\times 4\epsilon' k! n \right) \notag\\
& \le k\times \left(\left(2\epsilon^{1/2}n+541(\epsilon')^{1/2}k^3n^2\right)\times n
    + \left(\frac{k-1}{2k^2}n^2 - 541(\epsilon')^{1/2}k^3n^2\right)\times 4\epsilon' k! n\right) \notag\\
& \le 600(\epsilon')^{1/2}k^3n^3, \notag
\end{align}
to transform $\mathcal{H}$ into a subgraph of $\widehat{S}$.
This completes the proof of Lemma~\ref{LEMMA:k-partite-cancellative-3-graph-stability}.
\end{proof}%LEMMA
%%%%%%%%%%%%%%%%%%%%%%%%%%%%%%%

%%%%%%%%%%%%%%%%%%%%%%%%%%%%%%%%%
%%%%%%%%%%%%%%%%%%%%%%%%%%%%%%%%%%%%%%%
\section{Proof of Theorem~\ref{THM:local-max-cancelltive-3-fesible-region}}\label{SEC:feasible-region}
In this section we will prove the following statement which implies Theorem~\ref{THM:local-max-cancelltive-3-fesible-region}.

\begin{theorem}\label{THM:application-stability-left-is-small}
There exists an absolute constant $c_3>0$ such that
for every constant $\epsilon$ satisfying $0\le \epsilon \le c_3$
there exists $n_0$ such that the following holds for all $n\ge n_0$.
Every cancellative $3$-graph $\mathcal{H}$ on $n$ vertices with $|\partial\mathcal{H}| = (1-\epsilon)\frac{k-1}{2k}n^2$
satisfies $|\mathcal{H}| \le \frac{k-1}{6k^2}n^3 - \frac{k-1}{4k^2}\epsilon n^3+ O(\epsilon^{3/2}n^3)$.
\end{theorem}

{\bf Remarks.}
\begin{itemize}
\item[(a)]
If $\epsilon>0$ is sufficiently small, then
$|\mathcal{H}| \le \frac{k-1}{6k^2}n^3 - \frac{k-1}{4k^2}\epsilon n^3+ O(\epsilon^{3/2}n^3)
< \frac{k-1}{6k^2}n^3 - \frac{k-1}{6k^2}\epsilon n^3$.
This shows that $g(\mathcal{T}_{3},(k-1)/k-\epsilon) \le (k-1)/k^2 - \epsilon (k-1)/k^2$ for sufficiently small $\epsilon>0$.
\item[(b)]
The other part of Theorem~\ref{THM:local-max-cancelltive-3-fesible-region},
namely $g(\mathcal{T}_{3},(k-1)/k+\epsilon) \le (k-1)/k^2 - \delta$,
follows from inequality $(\ref{equ:fesible-region-cancel-3-right})$ and the fact that $x(1-x)$ is decreasing when $x>1/2$
(see Figure~\ref{fig:feasible-region-cancellative-3}).
\end{itemize}

Before proving Theorem~\ref{THM:local-max-cancelltive-3-fesible-region} let us present some useful lemmas.

\begin{lemma}\label{LEMMA:edges-in-S-blowup-3-graph}
Let $k\in 6\mathbb{N} + \{1,3\}$ and $k\ge 3$.
Let $\epsilon>0$ be a sufficiently small constant and $n$ be a sufficiently large constant.
Suppose that $\mathcal{H}$ is a $3$-graph on $n$ vertices with $|\partial\mathcal{H}| = (1-\epsilon)\frac{k-1}{2k}n^2$,
and $\mathcal{H}$ is a blowup of $\mathcal{S}$ for some $\mathcal{S} \in {\rm STS}(k)$.
Then $|\mathcal{H}| \le \frac{k-1}{6k^3}n^3 - \frac{k-1}{2k^2}\epsilon n^3 + 9\epsilon^{3/2}k^3n^3$.
\end{lemma}
\begin{proof}[Proof of Lemma~\ref{LEMMA:edges-in-S-blowup-3-graph}]
Let $V= V(\mathcal{H})$ and $V = V_{1}\cup \cdots \cup V_k$ be a partition such that $\mathcal{H}$ equals the blowup
$\mathcal{S}[V_1,\ldots,V_{k}]$ of $\mathcal{S}$.
Without loss of generality we may assume that $|V_1|\ge \cdots \ge |V_{k}|$.
Let $\delta = |V_{k}|/n - 1/k \ge 0$ and notice from Claim~\ref{claim-Vi-up-low-bound} that $\delta< 2\epsilon^{1/2}$.
For every $i\in[k]$ fix a subset $V_i'\subset V_i$ of size exactly $(1/k-\delta)n$ (note that $V_{k}' = V_{k}$).
Let $V' = V_1'\cup \cdots \cup V_{k}'$ and $R = V\setminus V'$.
Notice that $|V'| = k(1/k-\delta)n = n-k\delta n$ and $|R| = k\delta n$.
Since the induced subgraph $\mathcal{H}[V']$ is a balanced blowup of $\mathcal{S}$, we have
\begin{align}
|\mathcal{H}[V']| = \frac{1}{3}\binom{k}{2}\left(\frac{1}{k}-\delta\right)^3n^3. \notag
\end{align}
Since the induced subgraph of $\partial\mathcal{H}$ on $V'$ has size $\binom{k}{2}\left(1/k-\delta\right)^2n^2$
and every vertex in $R$ has exactly $(k-1)\left(1/k-\delta\right)n$ neighbors in $V'$,
the size of the induced subgraph of $\partial\mathcal{H}$ on $R$ satisfies
\begin{align}
|(\partial\mathcal{H})[R]|
= |\partial\mathcal{H}| - \frac{k-1}{2k}\left(\frac{1}{k}-\delta\right)^2n^2
     - k\delta n \cdot (k-1)\left(\frac{1}{k}-\delta\right)n
= \binom{k}{2}\delta^2 n^2 - \frac{k-1}{2k}\epsilon n^2. \notag
\end{align}
For $i\in\{1,2\}$ let $\mathcal{E}_{i}$ denote the set of edges in $\mathcal{H}$ that have exactly $i$ vertices in $V'$.
Notice that for every vertex $v\in R$ the induced subgraph of $L_{\mathcal{H}}(v)$ on $V'$ consists of $(k-1)/2$
balanced complete bipartite graphs and each of them has $2(1/k-\delta)n$ vertices.
Therefore,
\begin{align}
|\mathcal{E}_{1}|
= |R|\cdot \frac{k-1}{2}\left(\frac{1}{k}-\delta\right)^2 n^2
= \binom{k}{2}\left(\frac{1}{k}-\delta\right)^2 \delta n^3. \notag
\end{align}
On the other hand, since every pair $uv\in (\partial\mathcal{H})[R]$ satisfies $|N_{\mathcal{H}}(uv)\cap V'| = |V_{i}'| = (1/k-\delta)n$
for some unique $i\in [k]$, we obtain
\begin{align}
|\mathcal{E}_{2}|
= |(\partial\mathcal{H})[R]| \cdot \left(\frac{1}{k}-\delta\right) n
= \left(\binom{k}{2}\delta^2 - \frac{k-1}{2k}\epsilon\right)\left(\frac{1}{k}-\delta\right) n^3. \notag
\end{align}
Therefore,
\begin{align}
|\mathcal{H}|
& = |\mathcal{H}[V']| + |\mathcal{E}_{1}| + |\mathcal{E}_{2}| + |\mathcal{H}[R]| \notag\\
& = \frac{1}{3}\binom{k}{2}\left(\frac{1}{k}-\delta\right)^3n^3
    + \binom{k}{2}\left(\frac{1}{k}-\delta\right)^2 \delta n^3
    + \left(\binom{k}{2}\delta^2 - \frac{k-1}{2k}\epsilon\right)\left(\frac{1}{k}-\delta\right) n^3
    + (\delta k n)^3 \notag\\
& = \frac{k-1}{6k^3}n^3 - \frac{k}{6}\delta^3 n^3 -\frac{k-1}{2k^2}\epsilon n^3 + \frac{k-1}{2k}\epsilon\delta n^3 + \delta^3 k^3 n^3. \notag
\end{align}
Since $\delta < 2 \epsilon^{1/2}$, we obtain
$|\mathcal{H}| \le  \frac{k-1}{6k^3}n^3 - \frac{k-1}{2k^2}\epsilon n^3 + 9\epsilon^{3/2}k^3n^3$.
\end{proof}

The next lemma extends Lemma~\ref{LEMMA:edges-in-S-blowup-3-graph} from blowups of $\mathcal{S}$ to $\mathcal{S}$-colorable $3$-graphs.

\begin{lemma}\label{LEMMA:edges-in-S-colorable-3-graph}
Let $k\in 6\mathbb{N} + \{1,3\}$ and $k\ge 3$.
Let $\epsilon>0$ be a sufficiently small constant and $n$ be a sufficiently large constant.
Suppose that $(G,\mathcal{H})$ is a cancellative pair on a set $V$ of size $n$,
$V = V_{1}\cup \cdots \cup V_k$ is a partition such that $G$ is a $k$-partite graph with parts $V_1,\ldots,V_k$,
and $\mathcal{H}$ is a subgraph of the blowup $\mathcal{S}[V_1,\ldots,V_{k}]$ of some $\mathcal{S} \in {\rm STS}(k)$.
If $|G| = (1-\epsilon)\frac{k-1}{2k}n^2$,
then $|\mathcal{H}| \le \frac{k-1}{6k^2}n^3 - \frac{k-1}{4k^2}\epsilon n^3 + 2\epsilon^{3/2}n^{3}$.
\end{lemma}
\begin{proof}[Proof of Lemma~\ref{LEMMA:edges-in-S-colorable-3-graph}]
Let $\widehat{\mathcal{S}} = \mathcal{S}[V_1,\ldots,V_{k}]$.
Since $|G| = (1-\epsilon)\frac{k-1}{2k}n^2$,
it follows from Claim~\ref{claim-Vi-up-low-bound} that $|V_{i}| \ge n/k-2\epsilon^{1/2} n$ for every $i\in[k]$.
So $N_{\widehat{\mathcal{S}}}(uv) \ge n/k-2\epsilon^{1/2} n$ for every $uv\in \partial\widehat{\mathcal{S}}$.
Let $\delta \ge 0$ be the real number that satisfies
\begin{align}
|\partial\widehat{\mathcal{S}}| = (1-\delta)\frac{k-1}{2k}n^2. \notag
\end{align}
Note that $\delta \le \epsilon$ and $|\partial\widehat{\mathcal{S}}\setminus G| = (\epsilon-\delta)\frac{k-1}{2k}n^2$.
Define
\begin{align}
\mathcal{E} = \left\{E\in\widehat{\mathcal{S}}\colon \exists \{u,v\}\in \partial\widehat{\mathcal{S}}\setminus G
\text{ such that } \{u,v\}\subset E\right\}. \notag
\end{align}
Let us partition $\mathcal{E}$ into three subsets $\mathcal{E}_{1}$, $\mathcal{E}_{2}$, and $\mathcal{E}_{3}$,
where every edge $E \in \mathcal{E}_{i}$ satisfies that $|\binom{E}{2} \setminus G| = i$ for $i\in[3]$.
Since every pair $uv\in \partial\widehat{\mathcal{S}}$ satisfies
$N_{\widehat{\mathcal{S}}}(uv) \ge n/k-2\epsilon^{1/2} n$,
we have
\begin{align}
|\mathcal{E}_{1}| + 2|\mathcal{E}_{2}| + 3|\mathcal{E}_{3}|
= \sum_{uv\in \partial\widehat{\mathcal{S}}\setminus G}d_{\widehat{\mathcal{S}}}(uv)
& \ge |\partial\widehat{\mathcal{S}}\setminus G| \left(\frac{n}{k}-2\epsilon^{1/2} n\right) \notag\\
& =  (\epsilon-\delta)\frac{k-1}{2k}\left(\frac{1}{k}-2\epsilon^{1/2}\right)n^3. \notag
\end{align}
On the other hand, notice that
\begin{align}
|\mathcal{E}|
= |\mathcal{E}_{1}| + |\mathcal{E}_{2}| + |\mathcal{E}_{3}|
& \ge \frac{1}{2}\left(|\mathcal{E}_{1}| + 2|\mathcal{E}_{2}| + 3|\mathcal{E}_{3}|\right) - \mathcal{E}_{3}, \notag
\end{align}
and by inequality $(\ref{equ:fesible-region-cancel-3-left})$, we have
$|\mathcal{E}_{3}| \le \left(\partial\widehat{\mathcal{S}}\setminus G\right)^{3/2} \le \epsilon^{3/2}n^{3}$.
Therefore,
\begin{align}
|\mathcal{E}|
\ge \frac{1}{2}(\epsilon-\delta)\frac{k-1}{2k}\left(\frac{1}{k}-2\epsilon^{1/2}\right)n^3 - \epsilon^{3/2}n^{3}
\ge (\epsilon-\delta)\frac{k-1}{4k^2}n^3 - 2\epsilon^{3/2}n^{3}. \notag
\end{align}
We may view $\mathcal{H}$ as a $3$-graph obtained from $\widehat{\mathcal{S}}$ by removing some edges.
In particular, since $\mathcal{E} \subset \widehat{\mathcal{S}}\setminus \mathcal{H}$, we have
$|\mathcal{H}| \le |\widehat{\mathcal{S}}| - |\mathcal{E}|$.
It follows from Lemma~\ref{LEMMA:edges-in-S-blowup-3-graph} that
$|\widehat{\mathcal{S}}| \le \frac{k-1}{6k^2}n^3 - \frac{k-1}{2k^2}\delta n^3 + 9\delta^{3/2}k^3 n^3$.
Therefore,
\begin{align}
|\mathcal{H}|
\le |\widehat{\mathcal{S}}| - |\mathcal{E}|
& \le \frac{k-1}{6k^2}n^3 - \frac{k-1}{2k^2}\delta n^3 + 9\delta^{3/2}k^3 n^3
   - \left((\epsilon-\delta)\frac{k-1}{4k^2}n^3 - 2\epsilon^{3/2}n^{3}\right) \notag\\
& = \frac{k-1}{6k^2}n^3 - \frac{k-1}{4k^2}\epsilon n^3 + 2\epsilon^{3/2}n^{3}
       - \left(\frac{k-1}{4k^2}\delta n^3  - 9\delta^{3/2}k^3 n^3\right) \notag\\
& \le \frac{k-1}{6k^2}n^3 - \frac{k-1}{4k^2}\epsilon n^3 + 2\epsilon^{3/2}n^{3}. \notag
\end{align}
\end{proof}

The next lemma is the key in the proof of Theorem~\ref{THM:application-stability-left-is-small}.

\begin{lemma}\label{LEMMA:cancelltive-pair-G-k-partite-upper-bound}
Let $k\in 6\mathbb{N} + \{1,3\}$ and $k\ge 3$.
Let $\epsilon>0$ be a sufficiently small constant and $n$ be a sufficiently large constant.
Suppose that $(G,\mathcal{H})$ be a cancellative pair on $n$ vertices,
$\mathcal{G}$ is a $k$-partite graph, and $|\mathcal{G}| = (1-\epsilon)\frac{k-1}{2k}n^2$.
Then $|\mathcal{H}| \le \frac{k-1}{6k^2}n^3 - \frac{k-1}{4k^2}\epsilon n^3 + 10^9 \epsilon^{3/2}k^9 n^3$.
\end{lemma}
\begin{proof}[Proof of Lemma~\ref{LEMMA:cancelltive-pair-G-k-partite-upper-bound}]
Let $V = V(G) = V(\mathcal{H})$.
Let $V = V_1\cup \cdots \cup V_{k}$ be a partition such that $G$ is a subgraph of the
complete $k$-partite graph $\widehat{G}$ with parts $V_1,\ldots, V_{k}$.
Suppose to the contrary that
$|\mathcal{H}| \ge \frac{k-1}{6k^2}n^3 - \frac{k-1}{4k^2}\epsilon n^3 + 10^9 \epsilon^{3/2}k^9 n^3 > (1-2\epsilon)\frac{k-1}{6k^2}n^3$.
Then by Lemma~\ref{LEMMA:k-partite-cancellative-3-graph-stability},
$\mathcal{H}$ contains a subgraph $\mathcal{H}'$ of size at least $|\mathcal{H}| - \delta n^3$,
where $\delta = 600(2\epsilon)^{1/2}k^3 n^3$,
such that $\mathcal{H}'$ is a subgraph of the blowup $\widehat{\mathcal{S}} = \mathcal{S}[V_1,\ldots,V_{k}]$
of some $\mathcal{S} \in {\rm STS}(k)$.
We may assume that $\mathcal{H}' = \mathcal{H}\cap \widehat{\mathcal{S}}$
(otherwise we can replace $\mathcal{H}'$ by $\mathcal{H}\cap \widehat{\mathcal{S}}$).
Let $M_{G} = \widehat{G}\setminus G$ and call members in $M_{G}$ missing edges of $G$.
Note that
$$|M_{G}| = |\widehat{G}| - |G| \le \frac{k-1}{2k}n^2 - (1-\epsilon)\frac{k-1}{2k}n^2 < \frac{\epsilon}{2}n^2.$$
Define
\begin{align}
G_{\ell}'
= \left\{uv\in \partial\mathcal{H}'\colon d_{\mathcal{H}'}(uv) \ge \frac{n}{100k}\right\},
\quad{\rm and}\quad
\mathcal{H}''
= \left\{E\in \mathcal{H}'\colon \binom{E}{2}\subset G_{\ell}'\right\}. \notag
\end{align}

Let $\epsilon_1'\ge 0$ be the real number such that
$|G_{\ell}'| = (1-\epsilon_1')\frac{k-1}{2k}n^2$.
Note that $G_{\ell}' \subset G \subset \widehat{G}$ and  $\epsilon_1' \ge \epsilon$.

\begin{claim}\label{CLAIM:G-large-is-almost-all}
We have $|G_{\ell}'| \ge \frac{k-1}{2k}n^2 - 8\delta kn^2$.
In other words, $\epsilon_1' \le 16 \frac{k^2}{k-1} \delta$.
\end{claim}
\begin{proof}[Proof of Claim~\ref{CLAIM:G-large-is-almost-all}]
Suppose to the contrary that $|G_{\ell}'| < \frac{k-1}{2k}n^2 - 8\delta kn^2$.
Then
\begin{align}
3|\mathcal{H}'|
= \sum_{uv\in \partial\mathcal{H}'}N_{\mathcal{H}'}(uv)
& \le  |G_{\ell}'|\left(\frac{n}{k}+3\epsilon^{1/2}n\right) + |\partial\mathcal{H}'\setminus G_{\ell}'|\frac{n}{100k} \notag\\
& \le  \left(\frac{k-1}{2k}n^2 - 8\delta kn^2\right)\left(\frac{n}{k}+3\epsilon^{1/2}n\right) + 8\delta kn^2 \cdot \frac{n}{100k} \notag\\
& \le  \frac{k-1}{2k^2}n^3 + 3\epsilon^{1/2}n^3 - 8\delta kn^2\left(\frac{n}{k}-\frac{n}{100k}\right) \notag\\
& \le  \frac{k-1}{2k^2}n^3 - 5\delta n^3 + 3\epsilon^{1/2}n^3
< \frac{k-1}{2k^2}n^3 - 4\delta n^3, \notag
\end{align}
which contradicts $|\mathcal{H}'| \ge |\mathcal{H}| - \delta n^3$
and $|\mathcal{H}| \ge \frac{k-1}{6k^2}n^3 - \frac{k-1}{4k^2}\epsilon n^3 + 10^9 \epsilon^{3/2}k^9 n^3$.
\end{proof}

We will consider two cases depending on the value of $\epsilon_1'$.

%%%%%%%%%%%%%%%%%%%%%%%%%%%%%%%%%%%%%%%%%%%%%%%%
{\bf Case 1}: $\epsilon_{1}' > 400k^3\epsilon$.

Define
\begin{align}
\mathcal{B}_{1}
= \left\{E\in \mathcal{H}\setminus \mathcal{H}'\colon \binom{E}{2}\cap G_{\ell}' \neq\emptyset\right\}
\quad{\rm and}\quad
\mathcal{B}_{2}
= \left\{E\in \mathcal{H}\setminus \mathcal{H}'\colon \binom{E}{3}\subset G\setminus G_{\ell}'\right\}. \notag
\end{align}

\begin{claim}\label{CLAIM-B1-upper-bound}
We have $|\mathcal{B}_{1}| \le 50\epsilon kn^3$.
\end{claim}
\begin{proof}[Proof of Claim~\ref{CLAIM-B1-upper-bound}]
For every $uv\in G_{\ell}'$ let $\varphi(uv)\in[k]$ denote the index such that
$N_{\mathcal{H}'}(uv) \subset V_{\varphi(uv)}$.
Suppose that $E = \{u,v,w\}$ is contained in $\mathcal{B}_{1}$ and $uv\in G_{\ell}'$.
Then $w\not\in V_{\varphi(uv)}$, since otherwise $E$ would be contained in $\mathcal{H}'$.
Notice that every vertex  $w\in N_{\mathcal{H}}(uv) \setminus V_{\varphi(uv)}$ cannot be adjacent to
vertices in $V_{\varphi(uv)}$ (in $G$), since by Observation~\ref{OBS:cancellative-pair-co-neighbor-is-independent}
the set $N_{\mathcal{H}}(uv)$ is independent in $G$.
Therefore, $d_{M_{G}}(w) \ge V_{\varphi(uv)} \ge n/(100k)$ for every vertex $w\in N_{\mathcal{H}}(uv) \setminus V_{\varphi(uv)}$.
Let
\begin{align}
B_{V} = \bigcup_{uv\in G'} (N_{\mathcal{H}}(uv) \setminus V_{\varphi(uv)}). \notag
\end{align}
Since $|M_{G}| \le \frac{k-1}{2k}\epsilon n^2$,
it follows from $2|M_{G}| \ge \sum_{w\in B_{V}}d_{M_{G}}(w)$ and $d_{M_{G}}(w) \ge n/(100k)$ for every $w\in B_V$ that
$|B_{V}|\frac{n}{100k} \le 2 \times \frac{k-1}{2k}\epsilon n^2 < \epsilon n^2$.
Therefore, $|B_{V}| \le 100\epsilon k n$.

Since every edge in $\mathcal{B}_{1}$ contains at least one vertex in $B_{V}$,
it follows that $|\mathcal{B}_{1}| \le |B_{V}|\binom{n}{2} \le 50\epsilon k n^3$.
\end{proof}%CLAIM

Since $\partial\mathcal{B}_{2} \subset G\setminus G_{\ell}'$,
we have $|\partial\mathcal{B}_{2}| \le \frac{k-1}{2k}(\epsilon_1'-\epsilon) n^2 < \epsilon_1' n^2$.
So, by inequality $(\ref{equ:fesible-region-cancel-3-left})$,
$|\mathcal{B}_{2}| \le (\partial\mathcal{B}_{2})^{3/2} < (\epsilon_{1}')^{3/2} n^3$.
Therefore, Lemma~\ref{LEMMA:edges-in-S-colorable-3-graph} applied to $(G_{\ell}', \mathcal{H}'')$ yields
\begin{align}
|\mathcal{H}|
& = |\mathcal{H}''| + |\mathcal{H}'\setminus \mathcal{H}''| + |\mathcal{B}_{1}| + |\mathcal{B}_{2}| \notag\\
& \le \frac{k-1}{6k^2}n^3 - \frac{k-1}{4k^2}\epsilon_1' n^3 + 2(\epsilon_{1}')^{3/2} n^3 + \frac{k-1}{2k}\epsilon_{1}'n^2 \cdot \frac{n}{100k}
     + 50\epsilon kn^3 + (\epsilon_1')^{3/2}n^{3} \notag\\
& \le \frac{k-1}{6k^2}n^3 - \left(\frac{k-1}{4k^2}\epsilon_1' n^3 - \frac{k-1}{200k^2}\epsilon_1' n^3- (\epsilon_1')^{3/2}n^{3}
     + 50\epsilon kn^3\right) \notag\\
& \le \frac{k-1}{6k^2}n^3 - \frac{k-1}{4k^2}\epsilon n^3. \notag
\end{align}
where the last inequality follows from $\epsilon_1' > 400k^3 \epsilon$ and $\epsilon_1' \le 16 \frac{k^2}{k-1}\delta \ll 1$.

%%%%%%%%%%%%%%%%%%%%%%%%%%%%%%%
{\bf Case 2}: $\epsilon_{1}'\le 400k^3 \epsilon$.

Define
\begin{align}
G_{\ell}
= \left\{uv\in \partial\mathcal{H}\colon d_{\mathcal{H}}(uv) \ge \frac{n}{50k}\right\},
\quad{\rm and}\quad
\mathcal{H}_{\ell}
= \left\{E\in \mathcal{H}\colon \binom{E}{2} \subset G_{\ell}\right\}. \notag
\end{align}
Let $\epsilon_{1} \ge 0$ be the real numbers such that
$|G_{\ell}| = (1-\epsilon_1)\frac{k-1}{2k}n^2$.
Note that $G_{\ell} \subset G \subset \widehat{G}$ and  $\epsilon_1 \ge \epsilon$.

We claim that it suffices to show that
\begin{align}\label{equ:H-large-upper-bound}
|\mathcal{H}_{\ell}|
\le \frac{k-1}{6k^2}n^3 - \frac{k-1}{4k^2}\epsilon_{1}n^3 + 10^8 \epsilon_1^{3/2}k^9 n^3.
\end{align}
Indeed, suppose that $(\ref{equ:H-large-upper-bound})$ holds.
Then
\begin{align}
|\mathcal{H}|
& \le |\mathcal{H}_{\ell}| + |G\setminus G_{\ell}|\cdot \frac{n}{50k} \notag\\
& \le \frac{k-1}{6k^2}n^3 - \frac{k-1}{4k^2}\epsilon_{1}n^3 + 10^8 \epsilon_1^{3/2}k^9 n^3
         +  (\epsilon_1-\epsilon)\frac{k-1}{2k}n^2 \cdot \frac{n}{50k} \notag\\
& \le \frac{k-1}{6k^2}n^3 - \frac{k-1}{4k^2}\epsilon n^3
        - \left((\epsilon_1-\epsilon)\frac{k-1}{5k^2}-10^8 \epsilon_1^{3/2}k^9\right)n^3.  \notag
\end{align}
If $\epsilon_{1} \ge 2\epsilon$, then $(\epsilon_1-\epsilon)\frac{k-1}{5k^2}-10^8 \epsilon_1^{3/2}k^9> 0$.
If $\epsilon_{1} \le 2\epsilon$, then $10^8 \epsilon_1^{3/2}k^9 n^3 \le 10^9 \epsilon^{3/2}k^9 n^3$.
In either case we are done.

Let $\mathcal{H}_{\ell}' = \mathcal{H}_{\ell} \cap \mathcal{H}'$ and $\mathcal{H}_{\ell}'' = \mathcal{H}_{\ell}\cap \mathcal{H}''$.
Define
\begin{align}
\mathcal{B}_{1}'(a)
& = \left\{E\in \mathcal{H}_{\ell}\setminus \mathcal{H}_{\ell}'\colon \left|\binom{E}{2}\cap G_{\ell}'\right| \ge 2\right\}, \notag\\
\mathcal{B}_{1}'(b)
& = \left\{E\in \mathcal{H}'_{\ell}\setminus \mathcal{H}_{\ell}''\colon \left|\binom{E}{2}\cap G_{\ell}'\right| =2\right\}, \notag\\
\quad{\rm and}\quad
\mathcal{B}_{2}'
& = \left\{E\in \mathcal{H}_{\ell}\setminus \mathcal{H}_{\ell}''\colon
         \left|\binom{E}{2}\setminus G_{\ell}'\right| \ge 2\right\}. \notag
\end{align}

A crucial observation is that if an edge $E\in \mathcal{H}_{\ell}\setminus \mathcal{H}_{\ell}'' = \mathcal{H}_{\ell}\setminus \mathcal{H}''$
satisfies $\left|\binom{E}{2}\cap G_{\ell}'\right| = 3$, then
$E\in \mathcal{H}_{\ell}\setminus \mathcal{H}_{\ell}' = \mathcal{H}_{\ell}\setminus \mathcal{H}'$.
Indeed, if $E\in \mathcal{H}_{\ell}' \subset \mathcal{H}'$ and $\binom{E}{3}\subset G_{\ell}'$,
then by the definition of $\mathcal{H}''$, we would have $E\in \mathcal{H}''$,
which implies that $E\in \mathcal{H}_{\ell}\cap \mathcal{H}'' = \mathcal{H}_{\ell}''$.

Therefore, $\mathcal{B}_{1}'(a) \cup \mathcal{B}_{1}'(b) \cup \mathcal{B}_{2}'$
is a partition of $\mathcal{H}_{\ell}\setminus \mathcal{H}_{\ell}''$.

\begin{claim}\label{CLAIM-B1'(a)-upper-bound}
We have $|\mathcal{B}_{1}'(a)| \le 10^4 \epsilon^2 k^2 n^2$.
\end{claim}
\begin{proof}[Proof of Claim~\ref{CLAIM-B1'(a)-upper-bound}]
Similar to the proof of Claim~\ref{CLAIM-B1-upper-bound},
for every $uv\in G_{\ell}'$ let $\varphi(uv) \in [k]$ denote the index such that
$N_{\mathcal{H}'}(uv) \subset V_{\varphi(uv)}$.
Let
\begin{align}
B_{V} = \bigcup_{uv\in G'} (N_{\mathcal{H}}(uv) \setminus V_{\varphi(uv)}),  \notag
\end{align}
and recall from the proof of Claim~\ref{CLAIM-B1-upper-bound} that $|B_{V}| \le 100\epsilon k n$.

Suppose that $E = \{u,v,w\}$ is contained in $\mathcal{B}_{1}'(a)$ and assume that $uv,uw\in G_{\ell}'$.
Then by the definition of $B_{V}$, we have $w\in B_V$ (because $uv\in G_{\ell}'$) and $v\in B_V$ (because $uw\in G_{\ell}'$).
Therefore, $E$ has at least two vertices in $B_{V}$.
It follows that $|\mathcal{B}_{1}'(a)| \le \binom{|B_V|}{2}\cdot n < 10^4 \epsilon^2 k^2 n^2$.
%We claim that every edge in $\mathcal{B}_{1}'$ is contained in $B_{V}$.
%Indeed, suppose to the contrary that there exists an edge $E\in \mathcal{B}_{1}'$
%such that $E\not\subset B_V$.
%Let us assume that $E = \{u,v,w\}$ and $w\not\in B_V$.
%Then it follows from the definition of $B_V$ that $w\in V_{\varphi(uv)}$, which implies that $\{u,v,w\}\in \mathcal{H}'$
%contradicting the observation before Claim~\ref{CLAIM-B1'(a)-upper-bound} that $E\not\in\mathcal{H}'$.
%Therefore, every edge in $\mathcal{B}_{1}'$ is contained in $B_{V}$,
%and it follows that $|\mathcal{B}_{1}'| \le |B_{V}|^3 \le 10^{6}\epsilon^3 k^3 n^3$.
\end{proof}%CLAIM

Let $\epsilon_2\ge 0$ be the real number such that $|G_{\ell} \cap G_{\ell}'| = (1-\epsilon_{2})\frac{k-1}{2k}n^2$.
Notice that $|G_{\ell} \cap G_{\ell}'| \ge |G_{\ell}| - |G\setminus G_{\ell}'| \ge |G_{\ell}| - \epsilon_1'\frac{k-1}{2k}n^2$,
thus $\epsilon_{2} \le \epsilon_1 + \epsilon_1' < \epsilon_1+ 400k^3 \epsilon < 401k^3\epsilon_1$.

Now suppose that $E = \{u,v,w\}$ is contained in $\mathcal{B}_{1}'(b)$.
By definition, $E\cap (G_{\ell} \setminus G_{\ell}') \neq\emptyset$,
and without loss of generality we may assume that $uv \in G_{\ell} \setminus G_{\ell}'$.
Since $uv \in G_{\ell}\subset G_{\ell}'$, it follows from definitions of $G_{\ell}$ and $G_{\ell}'$ that
$|N_{\mathcal{H}}(uv)| \ge \frac{n}{50k}$ and $|N_{\mathcal{H}'}(uv)| < \frac{n}{100k}$.
This implies that there exists $i\in [k]\setminus \{\varphi(uv)\}$ such that
$|N_{\mathcal{H}}(uv) \cap V_{i}| \ge \frac{n}{50k} - \frac{n}{100k} = \frac{n}{100k}$.
By Claim~\ref{claim-cancel-3-N(uv)-concenrate}, we actually have
$|N_{\mathcal{H}}(uv) \cap V_{i}| \ge |N_{\mathcal{H}}(uv)| - \frac{\epsilon n^2}{|N_{\mathcal{H}}(uv)|}
\ge |N_{\mathcal{H}}(uv)| - 100\epsilon kn^2$,
which in turn implies that $|N_{\mathcal{H}'}(uv)| \le |N_{\mathcal{H}}(uv) \setminus V_i| \le 100\epsilon k n^2$.
Therefore,
\begin{align}
|\mathcal{B}_{1}'(b)|
\le \sum_{uv\in G_{\ell} \setminus G_{\ell}'}|N_{\mathcal{H}'}(uv)|
\le 100\epsilon k n^2 \cdot |G \setminus G_{\ell}'|
\le  100\epsilon\epsilon_1' k n^3. \notag
\end{align}
Now suppose that $E$ is an edge in $\mathcal{B}_{2}'$.
By definition $\left|\binom{E}{2}\cap (G_{\ell}\setminus G_{\ell}')\right| \ge 2$.
So
\begin{align}
|\mathcal{B}_{2}'|
\le \frac{1}{2}\sum_{uv\in G_{\ell}\setminus G_{\ell}'}N_{\mathcal{H}_{\ell}}(uv)
& \le \frac{1}{2}|G_{\ell}\setminus G_{\ell}'|\left(\frac{n}{k}+3\epsilon^{1/2}n\right) \notag\\
& =   \frac{1}{2}|G_{\ell}\setminus(G_{\ell} \cap G_{\ell}')|\left(\frac{n}{k}+3\epsilon^{1/2}n\right) \notag\\
& = \frac{1}{2}(\epsilon_2-\epsilon_1)\frac{k-1}{2k}n^2\left(\frac{n}{k}+3\epsilon^{1/2}n\right) \notag\\
& \le \frac{k-1}{4k^2}(\epsilon_2-\epsilon_1)n^3 + \epsilon^{1/2} \epsilon_2 n^3\notag
\end{align}
(the factor $1/2$ is due to the fact that every edge in $\mathcal{B}_{2}'$ is counted at least twice in the summation
$\sum_{uv\in G_{\ell}\setminus G_{\ell}'}N_{\mathcal{H}_{\ell}}(uv)$).

Now Lemma~\ref{LEMMA:edges-in-S-colorable-3-graph} applied to $(G_{\ell}\cap G_{\ell}',\mathcal{H}_{\ell}'')$ yields
\begin{align}
|\mathcal{H}_{\ell}|
& = |\mathcal{H}_{\ell}''| + |\mathcal{B}_{1}'(a)| + |\mathcal{B}_{2}'(b)| + |\mathcal{B}_{2}'| \notag\\
& \le \frac{k-1}{6k^2}n^3 - \frac{k-1}{4k^2}\epsilon_2n^3 + 2\epsilon_{2}^{3/2}n^3 + 10^4\epsilon^2 k^2n^3
     + \frac{k-1}{4k^2}(\epsilon_2-\epsilon_1)n^3 + \epsilon^{1/2} \epsilon_2 n^3 \notag\\
& \le \frac{k-1}{6k^2}n^3 - \frac{k-1}{4k^2}\epsilon_1n^3
     + \left(2\epsilon_{2}^{3/2}n^3+10^4\epsilon^2 k^2n^3+\epsilon^{1/2} \epsilon_2 n^3\right) \notag\\
& \le \frac{k-1}{6k^2}n^3 - \frac{k-1}{4k^2}\epsilon_1n^3
     + \left(2(401k^3\epsilon_1)^{3/2}+10^4(401k^3\epsilon_1)^2 k^2+\epsilon_{1}^{1/2} (401k^3\epsilon_1) \right)n^3 \notag\\
& \le \frac{k-1}{6k^2}n^3 - \frac{k-1}{4k^2}\epsilon_1n^3 + 10^8 \epsilon_1^{3/2}k^9 n^3, \notag
\end{align}
this proves $(\ref{equ:H-large-upper-bound})$.
\end{proof}%LEMMA

%%%%%%%%%%%%%%%%%%%%%%%%%%%%%%%%%%%%%%%%%
Now we are ready to prove Theorem~\ref{THM:application-stability-left-is-small}.

\begin{proof}[Proof of Theorem~\ref{THM:application-stability-left-is-small}]
Fix $C$ be a sufficiently large constant (the exact value of $C$ can be obtained from the last inequality of the proof).
Let $\epsilon > 0$ be a sufficiently small constant and let $n$ be a sufficiently large integer.
Let $\mathcal{H}$ be a cancellative $3$-graph on $n$ vertices with $|\partial\mathcal{H}| = (1-\epsilon)\frac{k-1}{2k}n^2$.
Suppose to the contrary that
\begin{align}
|\mathcal{H}|
\ge \frac{k-1}{6k^2} n^3 - \frac{k-1}{4k^2} \epsilon n^3 + C\epsilon^{3/2}n^3
\ge (1-2\epsilon)\frac{k-1}{6k^2} n^3. \notag
\end{align}
Let $V = V(\mathcal{H})$.
By Lemma~\ref{LEMMA:cancellative-3-graph-stability-decompose-vertex-set}, there exists a set
$U\setminus V$ of size at most $130(2\epsilon) k^4 n$ such that
the induced subgraph of $\partial\mathcal{H}$ on $V\setminus U$ is $k$-partite.
Viewing $\mathcal{H}[V\setminus U]$ as a subgraph on $V$  (so $U$ is a set of isolate vertices in $\mathcal{H}[V\setminus U]$).
Let $\epsilon_1\ge 0$ be the real number such that $|(\partial\mathcal{H})[V\setminus U]| = (1-\epsilon_{1})\frac{k-1}{2k}n^2$.
Notice that $|(\partial\mathcal{H})[V\setminus U]| \ge |\partial\mathcal{H}|-|U|n$,
so $\epsilon_1 \le 4\times 130(2\epsilon) k^4 \le 1100  k^4\epsilon$.
Since the pair $((\partial\mathcal{H})[V\setminus U], \mathcal{H}[V\setminus U])$ is cancellative,
it follows from Lemma~\ref{LEMMA:cancelltive-pair-G-k-partite-upper-bound} that
\begin{align}
|\mathcal{H}[V\setminus U]| \le \frac{k-1}{6k^2}n^3 - \frac{k-1}{4k^2}\epsilon_1 n^3 + 10^9 \epsilon_{1}^{3/2}k^9 n^3. \notag
\end{align}
For $i\in [3]$ let $\mathcal{E}_{i}$ denote the set of edges in $\mathcal{H}$ that have exactly $i$ vertices in $U$.
For every $u\in U$ let $d_{V\setminus U}(u) = |N_{\mathcal{H}}(u)\setminus U|$.
Similar to Claim~\ref{CLAIM:size-common-neighbors},
we have $|N_{\mathcal{H}}(uv)\setminus U| \le \left(1/k+3\epsilon_{1}^{1/2}\right)n$ for every $v\in N_{\mathcal{H}}(u)$.
In other words, every vertex in $L_{\mathcal{H}}(v)[V\setminus U]$ has degree at most $\left(1/k+3\epsilon_{1}^{1/2}\right)n$.
Therefore,
\begin{align}
|\mathcal{E}_{1}|
\le \sum_{u\in U} |L_{\mathcal{H}}(v)[V\setminus U]|
& \le \frac{1}{2}\sum_{u\in U} d_{V\setminus U}(u) \times \left(\frac{1}{k}+3\epsilon_{1}^{1/2}\right)n \notag\\
& = \frac{1}{2}|(\partial\mathcal{H})[V, U]|\left(\frac{1}{k}+3\epsilon_{1}^{1/2}\right)n. \notag
\end{align}
Therefore,
\begin{align}
|\mathcal{H}|
& \le |\mathcal{H}[V\setminus U]| + |\mathcal{E}_{1}| + |\mathcal{E}_{2}| + |\mathcal{E}_{3}| \notag\\
& \le \frac{k-1}{6k^2}n^3 - \frac{k-1}{4k^2}\epsilon_1 n^3 + 10^9 \epsilon_{1}^{3/2}k^9 n^3
   + \frac{1}{2}|(\partial\mathcal{H})[V,U]|\left(\frac{1}{k}+3\epsilon_{1}^{1/2}\right)n + \binom{|U|}{2} n \notag\\
& \le \frac{k-1}{6k^2}n^3 - \frac{k-1}{4k^2}\epsilon_1 n^3
   + \frac{1}{2}\frac{k-1}{2k}(\epsilon_1-\epsilon)n^2\left(\frac{1}{k}+3\epsilon_{1}^{1/2}\right)n
   + \left(10^9 \epsilon_{1}^{3/2}k^9 + (260)^2\epsilon^2 k^8\right) n^3\notag\\
& \le \frac{k-1}{6k^2}n^3 - \frac{k-1}{4k^2}\epsilon_1 n^3
   + \frac{k-1}{4k^2}(\epsilon_1-\epsilon)n^3 + \left(\epsilon_1^{3/2} + 10^9 \epsilon_{1}^{3/2}k^9 + (260)^2\epsilon^2 k^8\right)n^3\notag\\
& < \frac{k-1}{6k^2}n^3 - \frac{k-1}{4k^2}\epsilon n^3 + C \epsilon^{3/2} n^3 \notag
\end{align}
contradicting our assumption (the last inequality used $\epsilon_1 \le 1100  k^4 \epsilon$).
This completes the proof of Theorem~\ref{THM:application-stability-left-is-small}.
\end{proof}

%%%%%%%%%%%%%%%%%%%%%%%%%%%%%%%%%%%%%%%%%%%%%%%%%%%%%%%%%%%%%%%%%%%%%%%%%%%%%%
\section{Proof of Theorem~\ref{thm-cancellative-exact}}\label{SEC:exact-result}
In this section we will prove Theorem \ref{thm-cancellative-exact}.
The following lemma will be useful for the proof.

\begin{lemma}\label{lemma-cancel-3-exact-induction}
Suppose that $(G,\mathcal{H})$ is a cancellative pair on set $V$, and $V = S\cup T$ is a partition.
If $G[S]$ is a complete graph,
then $|\mathcal{H}| \le |\mathcal{H}[T]| + |G[T]| + |G[S,T]|/2  + |G[S]|/3$.
\end{lemma}
\begin{proof}[Proof of Lemma~\ref{lemma-cancel-3-exact-induction}]
It suffices to prove that the $|\mathcal{H}\setminus\mathcal{H}[T]| \le |G[T]| + |G[S,T]|/2  + |G[S]|/3$.
For $i\in[3]$ let $\mathcal{E}{i}$ denote the set of edges in $\mathcal{H}\setminus\mathcal{H}[T]$ that have
exactly $i$ vertices in $S$.

Since $G[S]$ is complete, it follows form Observation~\ref{OBS:cancellative-pair-co-neighbor-is-independent} that
every pair $uv\in G[T]$ satisfies $|N_{\mathcal{H}}(uv)\cap S| \le 1$.
Therefore, $|\mathcal{E}_{1}| \le |G[T]|$.

Next, we claim that every pair $uv\in G[S,T]$ is contained in at most one edge in $\mathcal{E}_{2}$.
Indeed, suppose to the contrary that there exist two distinct vertices $w_1,w_2\in U$
such that $\{u,v,w_1\},\{u,v,w_2\}\in \mathcal{E}_{2}$.
Then $|N_{\mathcal{H}}(uv)\cap U|\ge 2$ contradicting Observation~\ref{OBS:cancellative-pair-co-neighbor-is-independent}.
Therefore, $|\mathcal{E}_{2}| \le |G[S,T]|/2$.

Finally, $|\mathcal{E}_{3}| \le |G[S]|/3$ follows from the observation that every pair $uv\in G[S]$ can be contained in at most one edge
in $\mathcal{E}_{3}$.
Therefore, $|\mathcal{H}\setminus\mathcal{H}[T]| = |\mathcal{E}_{1}| +  |\mathcal{E}_{2}| + |\mathcal{E}_{3}|
\le |G[T]| + |G[S,T]|/2  + |G[S]|/3$.
\end{proof}

The proof of Theorem \ref{thm-cancellative-exact} follows a similar pattern as the proof of
Theorem~\ref{THM:cancel-3-stability-steiner-real-statement}, but instead of showing that $\partial\mathcal{H}$
contains very few number of copies of $K_{k+1}$, we will show that it is actually $K_{k+1}$-free.
Then due to Tur\'{a}n's theorem, $\partial\mathcal{H}\cong T_{2}(n,k)$.
Finally, repeating the argument (starting from Claim~\ref{claim-cancel-3-L(v)-disjoint-bipartite})
in the proof of Lemma~\ref{LEMMA:k-partite-cancellative-3-graph-stability} one can easily show that
$\mathcal{H}$ is a blowup of a Steiner triple system on $k$ vertices.

\begin{proof}[Proof of Theorem \ref{thm-cancellative-exact}]
Let $k\in 6\mathbb{N}+\{1,3\}$, $k\ge 3$, and $n$ be a sufficiently large integer.
To keep the calculations simple, let us assume that $n$ is a multiple of $k$.
Let $\mathcal{H}$ be a cancellative $3$-graph on $n$ vertices with $|\partial\mathcal{H}| = t_{2}(n,k) = (k-1)n^2/(2k)$
and $|\mathcal{H}|\ge s(n,k) = (k-1)n^3/(6k^2)$.

Our first step is to show that $\partial\mathcal{H}$ is $K_{k+1}$-free.
If $\omega(\partial\mathcal{H}) \le k$, then we are done.
So we may assume that $\omega(\partial\mathcal{H}) \ge k+1$.
Let $(S_1,\ldots,S_{t},R)$ be a maximal clique expansion of $\partial\mathcal{H}$
with threshold $k$ for some positive integer $t$.
Let $t'\le t$ be the integer such that $|S_{t'}| \ge k+1$ but $|S_{t'+1}|\le k$.
Due to Observation~\ref{OBS:max-clique-expansion}~(a), $|S_1| \ge \cdots \ge |S_{k}|$,
so $t'$ is well-defined.
Let $e=|\partial\mathcal{H}|$.
For every $i\in[t]$ let $e_{i}$ denote the number of edges in $(\partial\mathcal{H})[V\setminus(S_1\cup\cdots\cup S_{i-1})]$
that have at least one vertex in $S_{i}$.
Let $\Sigma_i = \sum_{j=1}^{i} e_{j}$ and $W_i = \sum_{j=1}^{i} |S_i|$ for $i\in[t]$.

\begin{claim}\label{claim-cancel-3-exact-W-t'-up-bound}
We have $W_{t'} < 100k(k+1)$ and hence, $t' \le |W_{t'}|/(k+1) <  100k$.
\end{claim}
\begin{proof}[Proof of Claim \ref{claim-cancel-3-exact-W-t'-up-bound}]
The proof is very similar to the proof of Claim~\ref{claim-stability-sum-Si-up-bound}
and in order to keep the proof short, we will omit some details in the calculations.

Let $x' = 2(e-\Sigma_{t'})/(n-W_{t'})^2$.
Then by Theorem~\ref{thm-calcel-3-induction-shadow} and Lemma~\ref{LEMMA:cancellative-clique-expansion},
\begin{align}\label{inequ-cancel-3-H-up-bound-1}
|\mathcal{H}|
& \le \frac{x'(1-x')}{6}(n-W_{t'})^3 + 3n^2 + t'e \notag\\
& = \frac{-2\Sigma_{t'}^2 +  \left( 4e-(n-W_{t'})^2 \right)\Sigma_{t'}+(n-W_{t'})^2e-2e^2}{3(n-W_{t'})} + 3n^2 + t'e \notag\\
& = \frac{-2(W_{t'}n-t'n)^2 +  \left( 4e-(n-W_{t'})^2 \right)(W_{t'}n-t'n)+(n-W_{t'})^2e-2e^2}{3(n-W_{t'})} + 3n^2 + t'e \notag\\
& \le \frac{-2n^2 (t')^2 +\left( n(n+W_{t'})^2-(n+3W_{t'})e \right)t' -(e-W_{t'}n)(2e-n^2-W_{t'}^2)}{3(n-W_{t'})} +3n^2.
\end{align}
Since $\omega_{i}\ge k+1$ for $i\in[t']$, $t' \le W_{t'}/(k+1)$.
Therefore, we may substitute $t' = W_{t'}/(k+1)$ into (\ref{inequ-cancel-3-H-up-bound-1}) and obtain
\begin{align}\label{inequ-cancel-3-H-up-bound-2}
|\mathcal{H}|
& \le \frac{(k+1)\left( -2(k+1)e^2 +\left((k+1)n^2 +(2k+1)W_{t'}n+(k-2)W_{t'}^2 \right)e \right)}{3(k+1)^2(n-W_{t'})} \notag\\
& \quad - \frac{\left( (k+1)(n^2+W_{t'}^2) -2W_{t'}n \right)kW_{t'}n}{3(k+1)^2(n-W_{t'})} +3n^2.
\end{align}
Substituting $e = (k-1)n^2 /(2k)$ into (\ref{inequ-cancel-3-H-up-bound-2}) we obtain
\begin{align}
|\mathcal{H}|
\le \frac{k-1}{6k^2}n^3
    - \frac{\left( (k+1)^2n^2 - k(k^3+2k^2-k+2)W_{t'}n + 2k^3(k+1)W_{t'}^2 \right) W_{t'}n}{6k^2(k+1)^2(n-W_{t'})} + 3n^2. \notag
\end{align}
If $W_{t'} \ge 100k(k+1)$, then the inequality above implies that $|\mathcal{H}| < (k-1)n^3/(6k^2)$, a contradiction.
Therefore, $W_{t'} < 100(k+1)$ and $t' \le W_{t'}/(k+1)<100k$.
\end{proof}

Next we show that $|R|$ is small.
For every $i\in[t]$ let $\omega_i = |S_i|$,
$R_i = V\setminus (S_1\cup \cdots \cup S_i)$ (note that $R_t = R$),
and $G_i$ be the induced subgraph of $\partial\mathcal{H}$ on $R_i$.

\begin{claim}\label{claim-cancel-3-Tt-up-bound}
We have $|R|<20k^2 n^{1/2}$.
\end{claim}
\begin{proof}[Proof of Claim \ref{claim-cancel-3-Tt-up-bound}]
Since $G_t = (\partial\mathcal{H})[R]$ is $K_k$-free, it follows from Tur\'{a}n's theorem that $|G_t| \le (k-2)|R|^2/(2(k-1))$.
For every $i\in[t]$ since every vertex in $R_{i}$ is adjacent to at most $\omega_i-1$ vertices in $S_{i}$,
we obtain $|G_{i-1}|\le |G_i|+(\omega_i-1)|R_{i}|+\binom{\omega_i}{2}$.
Therefore,
\begin{align}%\label{inequ-cancel-3-G-up-bound-1}
|\partial\mathcal{H}|
& \le \left(\sum_{i = 1}^{t'}\omega_i\right)n
      + \sum_{j = t'+1}^{t}\left( \left(\omega_j-1\right) \left( n- \left( \sum_{i = 1}^{j}\omega_i \right)\right) + \binom{\omega_j}{2}\right)
      + |G_{t}| \notag\\
& \le W_{t'}n + \sum_{i = 0}^{t-t'-1}\left(k-1\right)\left(n-W_{t'}-ik\right) + \frac{k-2}{2(k-1)}|R|^2 \notag\\
& \le W_{t'}n + (t-t')(k-1)(n-W_{t'}) -\frac{k(k-1)(t-t')(t-t'-1)}{2} + \frac{k-2}{2(k-1)}|R|^2. \notag
\end{align}
Since $t-t' = (n-W_{t'}-|R|)/k$, the inequality above and $|\partial\mathcal{H}| = (k-1)n^2/(2k)$ imply that
\begin{align}
\frac{k-1}{2k}n^2
& \le W_{t'}n + (k-1)(n-W_{t'})\frac{n-W_{t'}-|R|}{k} -\frac{k(k-1)}{2}\left(\frac{n-W_{t'}-|R|}{k}\right)^2 \notag\\
& \quad + k^2n + \frac{k-2}{2(k-1)}|R|^2 \notag\\
& < \frac{k-1}{2k}n^2 - \frac{1}{2k(k-1)}|R|^2 + \frac{2W_{t'}}{k}n+k^2n \notag\\
& < \frac{k-1}{2k}n^2 - \frac{|R|^2}{2k(k-1)} + 200k^2 n, \notag
\end{align}
which implies that $|R|< 20k^2 n^{1/2}$ (the last inequality used Claim \ref{claim-cancel-3-exact-W-t'-up-bound}).
\end{proof}

Notice that the upper bound for $|\mathcal{H}|$ in Theorem~\ref{thm-calcel-3-induction-shadow} has an error term $3m^3$,
which can be of order $n^2$ if $m$ is large. Our next claim improves this error term to $O(n)$.
For $i\in[t]$ let $x_{i} = 2|G_i| / |R_{i}|^2$ be the edge density of $G_i$.

\begin{claim}\label{claim-cancel-3-Hi-up-bound}
For every $t' \le i \le t$ we have
\begin{align}
|\mathcal{H}_{i}| \le \frac{x_{i}(1-x_{i})}{6} |R_{i}|^3 + k|R_{i}| + 1200k^4 n. \notag
\end{align}
\end{claim}
\begin{proof}[Proof of Claim \ref{claim-cancel-3-Hi-up-bound}]
We proceed by backward induction on $i$.
When $i = t$, by Theorem~\ref{thm-calcel-3-induction-shadow} and Claim~\ref{claim-cancel-3-Tt-up-bound},
\begin{align}
|\mathcal{H}_{t}|
\le \frac{x_{t}(1-x_{t})}{6} |R|^3 + 3 |R|^2
= \frac{x_{t}(1-x_{t})}{6} |R|^3 + 1200k^4 n. \notag
\end{align}
Now assume that the claim is true for some $i+1$ with $t'+1 \le i+1 \le t$
and we want to show that it is also true for $i$.
By Lemma \ref{lemma-cancel-3-exact-induction} and the induction hypothesis,
\begin{align}
|\mathcal{H}_{i}|
& \le |\mathcal{H}_{i+1}| + |G_{i+1}| + {e_{i+1}}/{2} \notag\\
& \le \frac{x_{i+1}(1-x_{i+1})}{6} |R_{i+1}|^3 + k |R_{i+1}| + 1200k^4 n + |G_{i+1}| + \frac{e_{i+1}}{2}. \notag
\end{align}
Let
\begin{align}
\Delta
  = \left( \frac{x_{i}(1-x_{i})}{6} |R_{i}|^3 + k|R_{i}| \right)
     - \left( \frac{x_{i+1}(1-x_{i+1})}{6} |R_{i+1}|^3 + k |R_{i+1}| + |G_{i+1}| + \frac{e_{i+1}}{2}\right), \notag
\end{align}
and it suffices to show that $\Delta \ge 0$.
Note that $|G_{i+1}| = |G_{i}| - e_{i+1}$ and $|R_{i+1}| = |R_{i}| - k$, so
\begin{align}
\Delta
& \ge  \frac{4e^2_{i+1}-\left( 8|G_i|-2\left(|R_{i}|-k\right)^2-3\left(|R_{i}|-k\right) \right)e_{i+1}}{6\left(|R_{i}|-k\right)} \notag\\
&       + \frac{\left( 2k|G_i|+(k-3)\left(|R_{i}|-k\right)|R_{i}| \right)|G_{i}|}{3\left(|R_{i}|-k\right)|R_{i}|}+k^2. \notag
\end{align}
We may substitute $e_{i+1} = (k-1)(|R_{i}|-k)+\binom{k}{2}$ into the inequality above and obtain
\begin{align}
\Delta
\ge  \frac{8k|G_i|^2 - 4\left(3k|R_i|-|R_i|-k^2-k\right)|R_i||G_i|}{12|R_i|(|R_i|-k)}
      + \frac{(2|R_i|^2 - |R_i| -k)(2|R_i|-k)(k-1)}{12(|R_i|-k)} + k^2. \notag
\end{align}
Note that  $8k|G_i|^2 - 4\left(3k|R_i|-|R_i|-k^2-k\right)|R_i||G_i|$ is decreasing in $|G_i|$ when
\begin{align}
|G_i| \le \frac{3k-1}{4k}|R_i|^2 - \frac{k+1}{4}|R_i|. \notag
\end{align}
On the other hand, since $G_{i}$ is $K_{k+1}$-free, it follows from Tur\'{a}n's theorem that $|G_{i}| \le (k-1)|R_i|^2/(2k)$.
So, we may substitute $|G_{i}| = (k-1)|R_i|^2/(2k)$ into the inequality above and obtain
$\Delta \ge {(11k+1)k}/{12}>0$.
Therefore,
\begin{align}
|\mathcal{H}_{i}| \le \frac{x_{i}(1-x_{i})}{6}|R_i|^3 + k|R_i|+ 1200k^4 n. \notag
\end{align}
\end{proof}

\begin{claim}\label{claim-cancel-3-G-Kk+1-free}
The graph $\partial\mathcal{H}$ is $K_{k+1}$-free.
\end{claim}
\begin{proof}[Proof of Claim \ref{claim-cancel-3-G-Kk+1-free}]
Recall that $t'$ is the largest integer such that $\omega_{t'}\ge k+1$ and $W_{t'} = \sum_{i=1}^{t'}\omega_{i}$.
So it suffices to show that $t' = 0$, i.e. $W_{t'} = \sum_{i=1}^{t'}\omega_{i} = 0$.
Suppose that this is not true, i.e. $W_{t'}>0$.
Then by Lemma~\ref{LEMMA:cancellative-clique-expansion} and Claim \ref{claim-cancel-3-Hi-up-bound},
\begin{align}
|\mathcal{H}|
\le |\mathcal{H}_{t'}| + t'e
< \frac{x_{t'}(1-x_{t'})}{6} |R_{t'}|^3 + 1201k^4 n + t'e. \notag
\end{align}
Let
\begin{align}\label{inequ-cancel-3-exact-Delta-1}
\Delta = \frac{k-1}{6k^2} n^3 - \left( \frac{x_{t'}(1-x_{t'})}{6} |R_{t'}|^3 + t'e + 1201k^4 n \right),
\end{align}
and by assumption we should have $\Delta\le 0$.
Substituting $x_{t'} = 2(e-\Sigma_{t'})/(n-W_{t'})^2$ and $|R_{t'}| = n-W_{t'}$ into (\ref{inequ-cancel-3-exact-Delta-1}) we obtain
\begin{align}
\Delta
\ge \frac{k-1}{6k^2} n^3
    - \left( \frac{-2\Sigma_{t'}^2 + \left( 4e-(n-W_{t'})^2 \right)\Sigma_{t'}
    + (n-W_{t'})^2 e-2e^2}{3(n-W_{t'})} + t'e + 1201k^4 n  \right). \notag
\end{align}
Similar to the proof of Claim \ref{claim-cancel-3-exact-W-t'-up-bound},
we may substitute $\Sigma_{t'} = \left( W_{t'} - t' \right)n$, $t' = W_{t'}/(k+1)$, and $e=(k-1)n^2/(2k)$
into the inequality above and obtain
\begin{align}
\Delta
\ge \frac{\left( (k+1)^2n^2 - k(k^3+2k^2-k+2)W_{t'}n +2k^3(k+1)W_{t'}^2 \right)W_{t'}n}{6k^2(k+1)^2 (n-W_{t'})} - 1201k^4 n, \notag
\end{align}
which is greater than $0$ when $n$ is sufficiently large, a contradiction.
\end{proof}

Since $\partial\mathcal{H}$ is $K_{k+1}$-free and $|G| = t_{2}(n,k)$,
by Tur\'{a}n's theorem, $G \cong T_{2}(n,k)$.
The rest part of the proof is basically repeating the argument (starting from Claim~\ref{claim-cancel-3-L(v)-disjoint-bipartite})
in the proof of Lemma~\ref{LEMMA:k-partite-cancellative-3-graph-stability}, and we omit it here.
\end{proof}
%%%%%%%%%%%%%%%%%%%%%%%%%%%%%%%%%%%%%%%%%%%%%%%%%%%%
%%%%%%%%%%%%%%%%%%%%%%%%%%%%%%%%%%%%%%%%
\section{Concluding remarks}\label{SEC:remarks}
We showed that the cancellative triple systems and the Steiner triple systems are closely related to each other
by proving that every cancellative triple system whose shadow density is close to $(k-1)/k$ and
edge density is close to $(k-1)/k^2$ is structurally close to a balanced blowup of some Steiner triple system on $k$ vertices
for every $k\in 6\mathbb{N}+\{1,3\}$ and $k\ge 3$.
Moreover, using this stability result we proved that the feasible region function $g(\mathcal{T}_3)$ of $\mathcal{T}_{3}$
has infinitely many local maxima.

$\bullet$
It would be interesting to explore whether there are similar relations between cancellative $r$-graphs and $r$-uniforms designs for $r\ge 4$.
The case $r= 4$ is of particular interest since the stability property and the maximum size of an $n$-vertex cancellative $4$-graph
are already well studied \cite{SI87,PI08,LMR2}.

$\bullet$
Theorem~\ref{THM:local-max-cancelltive-3-fesible-region}
shows that the point $\left(\frac{k-1}{k},\frac{k-1}{k^2}\right)$ is a local maximum of $g(\mathcal{T}_{3})$
for every $k\in 6\mathbb{N}+\{1,3\}$ and $k\ge 3$.
In particular, $g(\mathcal{T}_3)$ of $\mathcal{T}_{3}$ has infinitely many local maxima.
This indicates that the shape of the feasible region $\Omega(\mathcal{F})$ might be quite wild for general families $\mathcal{F}$.
On the other hand, a complete determination of $g(\mathcal{T}_{3})$ seems hard and it would be interesting to see whether
methods that are used to solve the clique density problem \cite{RA08,NI11,RE16} can be applied here.

$\bullet$
A nontrivial (i.e. the first coordinate is not on the boundary of the interval ${\rm proj}\Omega(\mathcal{F})$) 
local minimum (if it exists) of a feasible region function seems more mysterious.
It would be interesting to find an explicit nontrivial local minimum of $g(\mathcal{F})$ for some family $\mathcal{F}$.

%%%%%%%%%%%%%%%%%%%%%%%%%%%%%%%%%%%%%%%%%%%%%%%%%%%
\section{Acknowledgement}
We are very grateful to Dhruv Mubayi for inspiring discussions on this project,
and many thanks to Sayan Mukherjee for carefully reading the manuscript and
for providing many helpful suggestions.

%%%%%%%%%%%%%%%%%%%%%%%%%%%%%%%%%%%%%%%%%%%%%%%%%%%%
\bibliographystyle{abbrv}
\bibliography{feasible_region}

\begin{thebibliography}{10}

\bibitem{AES74}
B.~Andr\'{a}sfai, P.~Erd\H{o}s, and V.~T. S\'{o}s.
\newblock On the connection between chromatic number, maximal clique and
  minimal degree of a graph.
\newblock {\em Discrete Math.}, 8:205--218, 1974.

\bibitem{BO74}
B.~Bollob{\'a}s.
\newblock Three-graphs without two triples whose symmetric difference is
  contained in a third.
\newblock {\em Discrete Math}, 8(1):21--24, 1974.

\bibitem{CO13}
F.~N. Cole.
\newblock The triad systems of thirteen letters.
\newblock {\em Trans. Amer. Math. Soc.}, 14(1):1--5, 1913.

\bibitem{EG81}
G.~P. Egorychev.
\newblock The solution of van der {W}aerden's problem for permanents.
\newblock {\em Adv. in Math.}, 42(3):299--305, 1981.

\bibitem{FA81}
D.~I. Falikman.
\newblock Proof of the van der {W}aerden conjecture on the permanent of a
  doubly stochastic matrix.
\newblock {\em Mat. Zametki}, 29(6):931--938, 957, 1981.

\bibitem{HS55}
M.~Hall, Jr. and J.~D. Swift.
\newblock Determination of {S}teiner triple systems of order {$15$}.
\newblock {\em Math. Tables Aids Comput.}, 9:146--152, 1955.

\bibitem{KA66}
G.~Katona.
\newblock A theorem of finite sets.
\newblock In {\em Theory of graphs ({P}roc. {C}olloq., {T}ihany, 1966)}, pages
  187--207, 1968.

\bibitem{KE11}
P.~Keevash.
\newblock Hypergraph {T}ur\'{a}n problems.
\newblock {\em Surveys in combinatorics}, 392:83--140, 2011.

\bibitem{KE18}
P.~Keevash.
\newblock Counting {S}teiner triple systems.
\newblock In {\em European {C}ongress of {M}athematics}, pages 459--481. Eur.
  Math. Soc., Z\"{u}rich, 2018.

\bibitem{KM04}
P.~Keevash and D.~Mubayi.
\newblock Stability theorems for cancellative hypergraphs.
\newblock {\em J. of Combin. Theory. Ser. B}, 92(1):163--175, 2004.

\bibitem{KR63}
J.~B. Kruskal.
\newblock The number of simplices in a complex.
\newblock In {\em Mathematical optimization techniques}, pages 251--278. Univ.
  of California Press, Berkeley, Calif., 1963.

\bibitem{LIU19}
X.~Liu.
\newblock New short proofs to some stability theorems.
\newblock {\em arXiv preprint arXiv:1903.01606}, 2019.

\bibitem{LM19}
X.~Liu and D.~Mubayi.
\newblock A hypergraph {T}ur\'an problem with no stability.
\newblock accepted by Combinatorica.

\bibitem{LM19A}
X.~Liu and D.~Mubayi.
\newblock The feasible region of hypergraphs.
\newblock {\em J. Comb. Theory, Ser. B}, 148:23--59, 2021.

\bibitem{LMR1}
X.~Liu, D.~Mubayi, and C.~Reiher.
\newblock Hypergraphs with many extremal configurations.
\newblock submitted.

\bibitem{LMR2}
X.~Liu, D.~Mubayi, and C.~Reiher.
\newblock A unified approach to hypergraph stability.
\newblock In preparation.

\bibitem{NI11}
V.~Nikiforov.
\newblock The number of cliques in graphs of given order and size.
\newblock {\em Trans. Amer. Math. Soc.}, 363(3):1599--1618, 2011.

\bibitem{PI08}
O.~Pikhurko.
\newblock An exact {T}ur{\'a}n result for the generalized triangle.
\newblock {\em Combinatorica}, 28(2):187--208, 2008.

\bibitem{RA08}
A.~Razborov.
\newblock On the minimal density of triangles in graphs.
\newblock {\em Combin. Probab. and Comput.}, 17(4):603--618, 2008.

\bibitem{RE16}
C.~Reiher.
\newblock The clique density theorem.
\newblock {\em Ann. of Math.}, pages 683--707, 2016.

\bibitem{SH96}
J.~B. Shearer.
\newblock A new construction for cancellative families of sets.
\newblock {\em Electron. J. Combin.}, 3(1):3, 1996.

\bibitem{SI87}
A.~F. Sidorenko.
\newblock On the maximal number of edges in a homogeneous hypergraph that does
  not contain prohibited subgraphs.
\newblock {\em Mat. Zametki}, 41(3):433--455, 459, 1987.

\bibitem{TU41}
P.~Tur{\'a}n.
\newblock On an extermal problem in graph theory.
\newblock {\em Mat. Fiz. Lapok}, 48:436--452, 1941.

\bibitem{RW03}
R.~{Wilson}.
\newblock {The early history of block designs}.
\newblock {\em {Rend. Semin. Mat. Messina, Ser. II}}, 9:267--276, 2004.

\bibitem{WR73}
R.~M. Wilson.
\newblock Nonisomorphic {S}teiner triple systems.
\newblock {\em Math. Z.}, 135:303--313, 1973/74.

\end{thebibliography}
\end{document}